\documentclass[a4paper,11pt,twoside]{article}

\RequirePackage[colorlinks,citecolor=blue,urlcolor=blue]{hyperref}
\usepackage[title]{appendix}
\usepackage{xkeyval}
\usepackage[utf8]{inputenc}
\usepackage[T1]{fontenc}
\usepackage[english]{babel}
\usepackage{charter}
\usepackage{multirow}
\usepackage{enumerate}
\usepackage{bm}
\usepackage{verbatim}
\usepackage{indentfirst}
\usepackage{xcolor}
\usepackage[top=5.cm, bottom=5.0cm, left=2cm, right=2cm]{geometry}
\usepackage{authblk}
\usepackage{mathtools}
\allowdisplaybreaks
\usepackage{framed}
\usepackage{amsmath}
\usepackage{amsthm} 
\usepackage{amsfonts}
\usepackage{mathrsfs}   
\usepackage{amssymb}
\usepackage{dsfont}
\usepackage{bbm}

\usepackage{cases}
\usepackage{tcolorbox}
\usepackage[shortlabels]{enumitem}

\usepackage{array}
\usepackage{makecell}

\newcolumntype{P}[1]{>{\centering\arraybackslash}p{#1}}

\theoremstyle{plain}
\newtheorem{theorem}{Theorem}[section]
\newtheorem{lemma}[theorem]{Lemma}
\newtheorem{proposition}[theorem]{Proposition}

\theoremstyle{definition}
\newtheorem{definition}[theorem]{Definition}

\newtheorem{example}[theorem]{Example}

\theoremstyle{definition}

\newtcolorbox{activitybox}[1][]{%
    breakable,
    enhanced,
    colback=white,
    colframe=black,
    coltitle=white,
    #1
}

\numberwithin{equation}{section}
\numberwithin{figure}{section}
 \usepackage[nodayofweek]{datetime}

\newcommand{\pmu}{\partial_{\mu}}
\newcommand{\ptwomu}{\partial^2_{\mu}}
\newcommand{\pthreemu}{\partial^3_{\mu}}

\newcommand{\pnmu}{\partial^n_{\mu}}
\newcommand{\px}{\partial_x}
\newcommand{\ptwox}{\partial^2_x}

\newcommand{\rvlaw}[1][]{  {\mathscr{L}}{{(#1)}}  }

\newcommand{\couplaw}[1][]{  {\mu}^{X,N}_{#1}  }
\newcommand{\twocouplaw}[1][]{  {\mu}^{X,2N}_{#1}  }

\newcommand{\mclaw}[1][]{  {\mu}^{X}_{#1}  }

\newcommand{\nlaw}[1][]{  {\mu}^{N}_{#1}  }

\newcommand{\twonlaw}[1][]{  {\mu}^{2N}_{#1}  }
\newcommand{\twonlawone}[1][]{  {\mu}^{2N,(1)}_{#1}  }
\newcommand{\twonlawtwo}[1][]{  {\mu}^{2N,(2)}_{#1}  }

\newcommand{\nlawell}[2]{  {\mu}^{N_{#2}}_{#1}  }
\newcommand{\nlawoneell}[2]{  {\mu}^{N_{#2},(1)}_{#1}  }
\newcommand{\nlawtwoell}[2]{  {\mu}^{N_{#2},(2)}_{#1}  }

\newcommand{\nlawelli}[4]{  {\mu}^{N_{#2},{#3},{#4}}_{#1}  }
\newcommand{\nlawoneelli}[4]{  {\mu}^{N_{#2},(1),{#3},{#4}}_{#1}  }
\newcommand{\nlawtwoelli}[4]{  {\mu}^{N_{#2},(2),{#3},{#4}}_{#1}  }

\newcommand{\etaeulerlaw}[1][]{  {\mu}^{Z,N,h}_{\eta(#1)}  }
\newcommand{\eulerlaw}[1][]{  {\mu}^{Z,N,h}_{#1}  }

\newcommand{\eulertwonlawone}[1][]{  {\mu}^{Z,2N,(1),h}_{#1}  }
\newcommand{\eulertwonlawtwo}[1][]{  {\mu}^{Z,2N,(2),h}_{#1}  }

\newcommand{\etaeulertwonlawone}[1][]{  {\mu}^{Z,2N,(1),h}_{\eta(#1)}  }
\newcommand{\etaeulertwonlawtwo}[1][]{  {\mu}^{Z,2N,(2),h}_{\eta(#1)}  }

\newcommand{\eulernlawelli}[5]{  {\mu}^{Z,N_{#2},{#5},{#3},{#4}}_{#1}  }
\newcommand{\eulernlawoneelli}[5]{  {\mu}^{Z,N_{#2},(1),{#5},{#3},{#4}}_{#1}  }
\newcommand{\eulernlawtwoelli}[5]{  {\mu}^{Z,N_{#2},(2),{#5},{#3},{#4}}_{#1}  }

\newcommand{\eulernlawellh}[3]{  {\mu}^{Z,N_{#2},{#3}}_{#1}  }
\newcommand{\eulernlawoneellh}[3]{  {\mu}^{Z,N_{#2},(1),{#3}}_{#1}  }
\newcommand{\eulernlawtwoellh}[3]{  {\mu}^{Z,N_{#2},(2),{#3}}_{#1}  }

\newcommand{\claw}[1]{\mathscr L{#1}}

\newcommand{\bE}{\mathbb{E}}

\newcommand{\bN}{\mathbb{N}}
\newcommand{\bP}{\mathbb{P}}

\newcommand{\bR}{\mathbb{R}}

\newcommand{\cA}{\mathcal{A}}

\newcommand{\cC}{\mathcal{C}}
\newcommand{\cD}{\mathcal{D}}

\newcommand{\cF}{\mathcal{F}}

\newcommand{\cI}{\mathcal{I}}

\newcommand{\cK}{\mathcal{K}}

\newcommand{\cM}{\mathcal{M}}

\newcommand{\cP}{\mathcal{P}}

\newcommand{\cV}{\mathcal{V}}

\def \eps {\epsilon}

\newcommand{\lev}{\left\langle}
\newcommand{\rev}{\right\rangle}

\DeclarePairedDelimiter{\ceil}{\lceil}{\rceil}

\newcommand{\R}{\mathbb{R}}

\newcommand{\ud}{\mathrm{d}}

\newcommand{\HP}[1] 
    {\ensuremath{\mathscr{H}^{#1}}}

\renewcommand{\Xi}[1]{X_{i #1}}

\title{Antithetic multilevel sampling method for non-linear functionals of measure}

\author[1,2]{{\L}ukasz Szpruch \thanks{This was work has been supported by The Alan Turing Institute under the Engineering and Physical Sciences Research Council grant EP/N510129/1.}}
\author[1]{Alvin Tse}
\affil[1]{School of Mathematics, University of Edinburgh}
\affil[2]{The Alan Turing Institute, London}
\date{}

\begin{document}    
\selectlanguage{english}
\maketitle

\begin{abstract}
Let $\mu\in \mathcal{P}_2(\bR^d)$, where $\mathcal{P}_2(\bR^d)$ denotes the space of square integrable probability measures, and consider a Borel-measurable function $\Phi:\mathcal P_2(\mathbb R^d)\rightarrow \mathbb R $. In this paper we develop Antithetic Monte Carlo estimator (A-MLMC) for $\Phi(\mu)$, which achieves sharp error bound under mild regularity assumptions. The estimator takes as input the empirical laws $\mu^N = \frac1N \sum_{i=1}^{N}\delta_{X_i}$, where a) $(X_i)_{i=1}^N$ is a sequence of i.i.d samples from $\mu$ or b)  $(X_i)_{i=1}^N$ is a system of interacting particles (diffusions) corresponding to a McKean-Vlasov stochastic differential equation (McKV-SDE).  Each case requires a separate analysis. For a mean-field particle system, we also consider the empirical law induced by its Euler discretisation which gives a fully implementable algorithm. As by-products of our analysis, we establish a dimension-independent rate of uniform \textit{strong propagation of chaos}, as well as an $L^2$ estimate of the antithetic difference for i.i.d. random variables corresponding to general functionals defined on the space of probability measures. 
\end{abstract}



\section{Introduction}

The convergence of the empirical law $\mu^N$ to its limit $\mu$ for linear functionals of measure (i.e.  $F(\mu)=\int_{\bR^d}f(x)\mu(dx)$ for some function $f:\bR^d \rightarrow \bR$) is rather well understood in the literature. Indeed, $F(\mu^N)$ is an unbiased estimator of $F(\mu)$ and in the i.i.d. case, the classical central limit theorems provide sharp error bounds. However, for general non-linear functionals of measure $\Phi:\mathcal P_2(\mathbb R^d)\rightarrow \mathbb R $, $\Phi(\mu^N)$ is typically a biased estimator of $\Phi(\mu)$ and hence when seeking an optimal estimator, more sophisticated techniques are needed. For example, in the context of nested Monte Carlo estimators, with $F(\mu)=R(\int_{\bR^d}f(x)\mu(dx))$ and $R:\bR\rightarrow \bR$ being nonlinear, the multilevel Monte Carlo (MLMC) \cite{lemaire2017multilevel,giles2018decision} and antithetic multilevel Monte Carlo (A-MLMC)  \cite{giles2018multilevel} estimators are more efficient than $F(\mu^N)$. In this work, we study the general case of functionals of measure, which are sufficiently smooth in an appropriate sense. Most importantly, we do not rely on  specific structural assumptions imposed on $\Phi(\mu)$. Our goal is to find an estimator $\cA  $ that approximates $\Phi(\mu)$.  We are interested in sharp estimates of its corresponding mean-square error\footnote{We look at the mean-square error for simplicity, but a similar computation could be done to verify the Lindeberg  condition and produce CLT with an appropriate scaling.} 
$
\bE[(\Phi(\mu) - \cA)^2  ].
$
The multilevel Monte Carlo approach provides a very efficient strategy when one aims to find an implementable algorithm that achieves a sharp upper bound for the mean-square error for a given computational cost (in the i.i.d case, cost can be defined as the number of random numbers needed to be generated to compute $\cA$). We shall consider the analysis of mean-square error via antithetic Monte Carlo techniques in two different cases: i.i.d. samples and interacting diffusions.  It is also worth mentioning that the rates of convergence studied in this work are independent of the dimension, under sufficient regularity of the functions concerned. If we only assume that $\Phi$ is Lipschitz continuous with respect to the Wasserstein distance, i.e, there exists a constant $C>0$ such that
$
|\Phi(\mu) - \Phi(\nu)| \leq \, C W_2(\mu,\nu)$, for all $\mu,\nu \in \mathcal P_2(\bR^d),
$
one could bound $|\Phi(\mu) -  \bE\Phi(\mu_N)|$ by $\bE W_2(\mu,\mu_N)$. Consequently, following \cite{fournier2015rate} or \cite{dereich2013constructive}, the rate of convergence in the number of samples $N$  deteriorates as the dimension $d$ increases. We also refer the reader to recent works \cite{ambrosio2018pde,talagrand2018scaling,goldman2018large} that study the problem from the perspective of Monge-Amp\'{e}re PDEs.  On the other hand, recently,  authors in \cite{delarue2019master} (in Lemma 5.10) observed that if the functional $\Phi$ is twice-differentiable with respect to the functional derivative (see Appendix \ref{ linear functional derivatives and L-derivatives appendix} for its definition), then one can obtain a dimension-independent bound for the strong error $\bE|\Phi(\mu) -  \Phi(\mu_N)|^4$, which is of order $O(N^{-1/2})$ (as expected by CLT).

\paragraph{Recap of MLMC.}
Fix $\Phi:\mathcal P_2(\mathbb R^d)\rightarrow \mathbb R $ and $L>0$.  We seek to approximate $\Phi(\mu)$.  In order to do so, we approximate measure $\mu$ with measures corresponding to different levels $\ell  \in \{0, \ldots, L \}$, where each $\ell$ corresponds to some level of approximation. As $\ell$ increases, the bias decreases and the corresponding computational cost increases. One therefore faces an optimisation problem and tries to obtain the minimal computation cost for a given accuracy (or, equivalently, to minimise the bias for a fixed computational budget).  It turns out, perhaps surprisingly, that the MLMC estimator, which consists of a hierarchy of biased approximations, can achieve computational efficiency (or optimal scaling of variance) of vanilla Monte Carlo built from directly accessible unbiased samples \cite{giles2008multilevel,giles2015multilevel}. 

{Let us consider a sequence of probability spaces $(\Omega^{\ell},\mathcal F^{\ell},\mathbb P^{\ell} )_{\ell=1}^{L}$. For each level $\ell$, we consider $M_{\ell}$ copies of $(\Omega^{\ell},\mathcal F^{\ell},\mathbb P^{\ell} )$, that is $(\Omega^{\ell,\theta},\mathcal F^{\ell,\theta},\mathbb P^{\ell.\theta} )_{\theta=1}^{M_{\ell}}$.  For each level $\ell \in \{0, \ldots, L\}$ and  cloud $\theta \in \{1, \ldots, M_{\ell} \}$, we generate $N_{\ell}$ samples (depending on the two cases: i.i.d. or interacting diffusion) and construct  empirical measure $\mu^{N_\ell,\theta,\ell} $.} For simplicity, we set $N_{\ell} = 2^{\ell}$ and let  $(M_{\ell})_{\ell=0}^L$ be a sequence of  non-increasing natural numbers. The MLMC estimator is defined as 
    \begin{eqnarray} 
        {\cA}^{\text{MLMC}}
        :=  \frac{1}{M_0} \sum_{\theta=1}^{M_0} \Phi(\mu^{N_0,\theta,0}) 
        + \sum_{\ell=1}^L  \bigg[ \frac{1}{M_{\ell}} \sum_{\theta=1}^{M_{\ell}} \Big[ \Phi(\mu^{N_\ell,\theta,\ell}) -  \Phi(\mu^{N_{\ell-1},\theta,\ell}) \Big] \bigg] . \label{eq MLMC}
    \end{eqnarray}
    Note that, for $\ell \in \{1, \ldots, L \}$, 
    $ \bE  \big[ \Phi (\mu^{N_{\ell-1},\theta,\ell}) \big] =  \bE  \big[  \Phi (\mu^{N_{\ell-1},\theta,\ell-1})\big] $.  
    Therefore, taking the expectation on both sides of \eqref{eq MLMC} gives a telescopic sum, which simplifies to 
    $$ \bE  \big[  {\cA}^{\text{MLMC}} \big]  = \bE  \big[  \Phi (\mu^{N_{L},1,L}) \big].$$ 
    If the variance between successive approximations converges to zero as the level increases, MLMC reduces the computational cost of simulation by carefully combining many simulations on low levels with low accuracy (at a corresponding low cost); with relatively few simulations on high levels with low accuracy (and at a high cost). The idea has been independently developed by Giles and Heinrich \cite{giles2015multilevel,heinrich2001multilevel,kebaier2005statistical} (see also 2-level Monte Carlo of Kebaier \cite{kebaier2005statistical}) in the context of temporal approximation of SDEs and parametric integration. 
    
{The key challenge in developing MLMC estimators is the construction of suitable coupling between $(\mu^{N_\ell,\theta,\ell},\mu^{N_{\ell-1},\theta,\ell})$ that ensures a quick decay of the variance of the estimator across the levels. With this in mind, in this work, we develop an antithetic extension of the MLMC algorithm (A-MLMC) defined by}
 \begin{eqnarray}
        {\cA}^{\text{A-MLMC}}
       & := & \frac{1}{M_0} \sum_{\theta=1}^{M_0} \Phi(\mu^{N_0,\theta,0}) \nonumber \\
       && + \sum_{\ell=1}^L \bigg[ \frac{1}{M_{\ell}} \sum_{\theta=1}^{M_{\ell}} \Big[ \underbrace{ \Phi(\mu^{N_\ell,\theta,\ell})}_{:=\Phi^{f,\ell} } - \underbrace{\frac{1}{2} \Big( \Phi(\mu^{N_\ell,(1),\theta,\ell}) + \Phi(\mu^{N_\ell,(2),\theta,\ell}) \Big) }_{:=\Phi^{c,\ell}} \Big] \bigg], 
       \label{eq AMLMC} 
    \end{eqnarray}
    where $\mu^{N_\ell,(1),\theta,\ell}$ and $\mu^{N_\ell,(2),\theta,\ell}$ correspond respectively to empirical measures corresponding to the first half and second half of $N_{\ell}$ samples (to be defined precisely later for the two cases separately), such that 
    $\bE \big[\Phi(\mu^{N_\ell,\theta,\ell}) \big] = \bE \big[ \Phi(\mu^{N_\ell, (1), \theta,\ell}) \big] =\bE \big[ \Phi(\mu^{N_\ell,(2), \theta,\ell}) \big] $. Let $\text{Cost} \big(  \Phi^{f,\ell} - \Phi^{c,\ell} \big)$ denote the computational cost of computing the estimator $\Phi^{f,\ell} - \Phi^{c,\ell}$. Theorem 1 in \cite{cliffe2011multilevel} gives a result concerning the complexity of antithetic MLMC.
    Let $\alpha,  \beta,  \gamma\in \mathbb R$ be positive constants such that $\alpha \geq \frac{1}{2} \min \{ \beta , \gamma \} $.  Suppose that
    \begin{equation} \label{eq mlmc condtions}
    i)\,\, \big| \bE \big[  \Phi^{f,\ell} - \Phi^{c,\ell} \big] \big| \lesssim N_{\ell}^{-\alpha}, 
    \quad ii) \,\, \text{Var}\big[  \Phi^{f,\ell} - \Phi^{c,\ell} \big] \lesssim N_{\ell}^{-\beta},     
    \quad iii)\,\,  \text{Cost} \big(  \Phi^{f,\ell} - \Phi^{c,\ell} \big)\lesssim N_{\ell}^{\gamma} \,.
    \end{equation}
    Then there exist $L$ and a sequence  $(M_{\ell})_{\ell=1}^{L}$ such that estimator \eqref{eq AMLMC} has a bound 
    \[
    \mathbb E[ (   \Phi(\mu) - {\cA}^{\text{A-MLMC}}  )^2  ] \lesssim \epsilon^2\,,
    \]
    with a computational {cost of ${\cA}^{\text{A-MLMC}}$ of order $\eps^{-2}$ provided $\beta>\gamma$. The core part of this work is to establish that indeed the condition $\beta>\gamma$ holds. We also remark that $\eps^{-2}$ is the computational cost of a Monte Carlo estimator that achieves mean-square-error of the order $\eps^2$ and has access to unbiased i.i.d. samples}. 

\emph{Notations.} $\quad $ Throughout this article, we denote the Hilbert-Schmidt norm of any matrix by $\| \cdot \|$ and denote the standard Euclidean inner product $x \cdot y $ by $xy$. Also, $\rvlaw[{\xi}]$ denotes the law of $\xi$, for any square-integrable random variable $\xi$.  $\cP_k(\bR^d) $ denotes the set of probability measures with finite $k$th moment.   $C^k_{b, \text{Lip}}((\bR^d)^{\ell})$ denotes the set of all functions from $(\bR^d)^{\ell}$ to $\bR$  that are in $C^k$ with bounded and Lipschitz partial derivatives up to and including order $k$. For any $a,b \geq 0$, we denote by $a \lesssim b$ if $a \leq Cb$, for some constant $C>0$  that does not depend on $N$, $h$ or $\eps$. Finally, unless otherwise specified, $C$ denotes a generic constant that does not depend on $N$, $h$ or $\eps$, whose value may vary from line to line. For any probability space ($\Omega, \cF, \bP)$ with expectation operator $\bE$ and a random variable $\xi$, we construct a separate probability space ($\widetilde{\Omega}, \widetilde{\cF},\widetilde{ \bP})$ with expectation operator $\widetilde{\bE}$ and a random variable $\widetilde{\xi}$ with the same law as $\xi$.

\section{Main results}
\subsection{A-MLMC for i.i.d samples} For every $ 0 \leq \ell \leq L$ and $ 1 \leq \theta \leq  M_{\ell} $, we consider i.i.d. random variables $\{X_{i,\ell,\theta}\}_{1 \leq i \leq N_{\ell}}$ with law  $\mu\in \cP_2(\bR^d)$. The corresponding antithetic MLMC estimator for $\Phi(\mu)$ is given by \eqref{eq AMLMC}, where  $$
\mu^{N_{\ell},\theta,\ell}:= \frac{1}{N_{\ell}} \sum_{i=1}^{N_{\ell}} \delta_{X^{}_{i,\ell,\theta}}$$ is the empirical measure corresponding to each of the $\sum_{\ell=0}^L M_{\ell}$ independent {clouds} of samples indexed by $\theta \in \{ 1, \ldots, M_{\ell} \}$ and $\ell \in \{0, \ldots, L\}$. Moreover, 
 $$ \mu^{N_\ell,(1), \theta,\ell}:= \frac{1}{N_{\ell}/2} \sum_{i=1}^{N_{\ell}/2} \delta_{X_{i,\ell, \theta}}, \quad \quad \quad \mu^{N_\ell,(2), \theta,\ell}:= \frac{1}{N_{\ell}/2} \sum_{i=N_{\ell}/2+1}^{N_{\ell}} \delta_{X_{i, \ell,\theta}}.$$
 
We illustrate the power of antithetic MLMC for i.i.d. samples through the following example in which there is a linear dependence on the measure in $\Phi$. For simplicity of notations, we set
\begin{equation} \mu^{N}:= \mu^{N,1,0}, \quad \quad \mu^{N,(1)}:= \mu^{N,(1),1,0},  \quad \quad \mu^{N,(2)}:= \mu^{N,(2),1,0}.
\label{empirical measures simplified} 
\end{equation}
 \begin{example}
Consider $\Phi(\mu):=\int_{\bR^d}F(x)\mu(dx)$, where $F:\bR^d \rightarrow \bR$ has linear growth. We already observed that  $\bE[{\cA}^{\text{MLMC}}]= \bE[ {\cA}^{\text{A- MLMC}}]$. The postulated independence conditions imply that
 \[
 \mathbb Var[{\cA}^{\text{MLMC}}] =\frac{\mathbb Var[\Phi(\mu^{N_0})]}{M_0} 
 + \sum_{\ell=1}^L \frac{\mathbb Var[\Phi(\mu^{N_{\ell}})-\Phi(\mu^{N_{\ell-1}})]}{M_{\ell}}\,. 
 \]
 On the other hand, 
  \[
 \mathbb Var[{\cA}^{\text{A-MLMC}}] =\frac{\mathbb Var[\Phi(\mu^{N_0})]}{M_0} 
 + \sum_{\ell=1}^L \frac{\mathbb Var\Big[ \Phi(\mu^{N_{\ell}}) - \frac{1}{2} \big( \Phi(\mu^{N_{\ell},(1)}) + \Phi(\mu^{N_{\ell},(2)}) \big) \Big] }{M_{\ell}}\,. 
 \]
 It is clear that the efficiency of this algorithm hinges on good coupling estimates that result in small variances across levels $\ell$. Set $N_{\ell}:= 2N_{\ell-1}$.
 For ${\cA}^{\text{MLMC}}$, we have
\begin{equation*}
\begin{split}
    \mathbb Var[\Phi(\mu^{N_{\ell}}) & -\Phi(\mu^{N_{\ell-1}})]  =   
   \mathbb Var \bigg[ \bigg( \frac{1}{N_{\ell}}- \frac{1}{N_{\ell-1}}\bigg) \sum_{i=1 }^{N_{\ell-1}}F(X_i) + \frac{1}{N_{\ell}} \sum_{i=N_{\ell-1}+1 }^{N_{\ell}}F(X_i) \bigg]   \\
   & =  \bigg( \frac{1}{N_{\ell}}- \frac{1}{N_{\ell-1}}\bigg)^2 \sum_{i=1 }^{N_{\ell-1}}  \mathbb Var [F(X_i)] + \bigg( \frac{1}{N_{\ell}}  \bigg)^2 \sum_{i=N_{\ell-1}+1 }^{N_{\ell}} \mathbb Var [F(X_i)]   
    =  O(\frac{1}{N_{\ell}}).  
\end{split} 
\end{equation*}
On the other hand, for A-MLMC, we have 
\[
\mathbb Var\Big[ \Phi(\mu^{N_{\ell}}) - \frac{1}{2} \big( \Phi(\mu^{N_{\ell},(1)}) + \Phi(\mu^{N_{\ell},(2)}) \big) \Big] =0\,.
\]
 \end{example}
 \subsubsection{Complexity analysis of A-MLMC for i.i.d. samples} \label{complexity iid subsection} 
 The above example is indeed a very special case. This work explores regularity conditions of functionals $\Phi$ that lead to a reduction in variance of the antithetic difference for general functions of measures. This result is formulated in terms of  the class $\cM^L_k$ of $k$ times differentiable functions in the sense of linear functional derivatives. (See Definition \ref{eq: def MLk} for its precise meaning. See also Definition \ref{eq:classmk} for the class $\cM_k$ of $k$ times differentiable functions in the sense of L-derivatives that will be used in other theorems). 

Recall that, from Section 9 in \cite{giles2015multilevel},  the  second moment of the antithetic difference  given by
$$ U (\twonlaw[]) -  \frac{1}{2} \big( U (\twonlawone[]) + U (\twonlawtwo[]) \big) $$ 
converges to $0$ in the rate $O(1/N^2)$, for functions $U: \cP_2(\bR^d) \to \bR$ of the form
\begin{equation}  U(\mu):= F \bigg( \int_{\bR^d} G(x) \, \mu(dx) \bigg), \label{eq: giles antithetic special form} \end{equation}
where $G: \bR^d \to \bR$ is an integrable function and $F: \bR\to \bR$ is a twice-differentiable function with bounded derivatives.  The following theorem generalises the structure of $U$ from the aforementioned result, under a set of different conditions.  
{
\begin{theorem}[Theorem \ref{antithetic initial separation}, antithetic error on i.i.d. random variables] \label{anti initial intro}  Suppose that $\mu \in \cP_8(\bR^d)$ and $ U \in \cM^L_4(\cP_2(\bR^d))$.
Then there exists a constant $C>0$ such that 
$$ \bE \big| U (\twonlaw[]) -  \frac{1}{2} \big( U (\twonlawone[]) + U (\twonlawtwo[]) \big) \big|^2 \leq \frac{C}{N^2}. $$
\end{theorem}
The proof of the theorem can be found in Section \ref{sec iid}. 
 }
 Theorem  \ref{antithetic initial separation} allow us to establish bounds on the variance in \eqref{eq mlmc condtions}, that is we have
\begin{equation}
\mathbb Var\Big[ \Phi(\mu^{N_{\ell}}) - \frac{1}{2} \big( \Phi(\mu^{N_{\ell},(1)}) + \Phi(\mu^{N_{\ell},(2)}) \big) \Big] =O(\frac{1}{N_{\ell}^2}). \label{iid mlmc ii}
\end{equation}
By Theorem 2.11 in \cite{chassagneux2019weak}, under the same hypothesis on $\Phi$, we also have  \begin{equation}
| \bE[\Phi(\mu^{N_{\ell}})] - \Phi(\mu) | \leq O(\frac{1}{N_{\ell}}). 
\label{iid mlmc i}
\end{equation}
Finally, since the empirical measures $\mu^{N_{\ell}}$, $\mu^{N_{\ell},(1)}$ and $\mu^{N_{\ell},(2)}$ correspond to i.i.d. random variables, the cost of simulating the antithetic difference is given by
\begin{equation}
 \text{Cost} \Big[ \Phi(\mu^{N_{\ell}}) - \frac{1}{2} \big( \Phi(\mu^{N_{\ell},(1)}) + \Phi(\mu^{N_{\ell},(2)}) \big) \Big] =O({N_{\ell}}). \label{iid mlmc iii}
\end{equation}
Hence, by combining \eqref{iid mlmc ii}, \eqref{iid mlmc i} and \eqref{iid mlmc iii}, along with Theorem 1 in \cite{cliffe2011multilevel}, we have the following result regarding the complexity of the antithetic MLMC estimator.
\begin{theorem} \label{thm A mlmc complexity} 
Let $\mu \in \cP_8(\bR^d)$ and $\Phi \in \cM^L_4$. Then, for the  mean-square error $
\bE[(\Phi(\mu) - {\cA}^{\text{A-MLMC}})^2  ]
$ to be  of the order $O(\eps^2)$, there exist $L$ and a sequence $(M_{\ell})_{\ell=1}^L$ such that the computational cost of 
${\cA}^{\text{A-MLMC}}$ is of the order $O(\eps^{-2})$.
\end{theorem}

%

\subsection{A-MLMC for interacting diffusions}

The second situation we treat in this work concerns estimates of propagation-of-chaos type for the system of McKean-Vlasov SDEs (McKV-SDEs). Building on regularity results 
recently obtained in \cite{chassagneux2019weak}, we extend the analysis of the i.i.d. case presented above to interacting particle systems. To be more precise, fix $T>0$ and let $\{W_t\}_{t\in [0,T]}$ be a $d$-dimensional Brownian motion on a filtered probability space $(\Omega, \{ \cF_t \}_{t}, \mathcal{F}, \bP)$. Next, we consider functions $b: \mathbb{R}^{d} \times \cP_2(\mathbb{R}^{d} ) \rightarrow \mathbb{R}^d$, $\sigma: \mathbb{R}^{d} \times \cP_2(\mathbb{R}^{d} ) \rightarrow \mathbb{R}^{d} \otimes \bR^{d}$, and consider the corresponding McKV-SDE given by
\begin{equation} \label{eq:MVSDE} 
\begin{cases} 
      dX_t = \xi + \int_0^t b(X_s, \mclaw[s]) \,ds + \int_0^t \sigma (X_s, \mclaw[s]) \,dW_s, \quad \quad t \in [0, T], \\
       \mclaw[s] := \text{Law} (X_s), 
   \end{cases}
\end{equation}
where $\xi \sim \nu \in \cP_2( \bR^d)$. The theory of propagation of chaos, \cite{sznitman1991topics}, shows that \eqref{eq:MVSDE} arises as a limiting equation  of the system of interacting diffusions (particles) $\{ Y^{i,N}_t \}_{i=1, \ldots, N}$ on $(\bR^d)^N$ given by  
\begin{equation} \label{eq:particlesystem} \begin{cases} 
      Y^{i,N}_t = \xi_i + \int_0^t  b(Y^{i,N}_s, \nlaw[s]) \,ds + \int_0^t \sigma(Y^{i,N}_s, \nlaw[s]) \,dW^i_s, \quad 1 \leq i \leq N, \quad t \in [0,T], \\
       \nlaw[s] := \frac{1}{N} \sum_{i=1}^N \delta_{Y^{i,N}_s}, 
   \end{cases}
\end{equation}
  where $W^i, 1 \leq i \leq N,$ are independent $d$-dimensional Brownian motions and $\xi_i,  1 \leq i \leq N,$ are i.i.d. random variables  with law $\nu \in \cP_2(\bR^d)$.  We refer the reader to \cite{gartner1988mckean,sznitman1991topics,meleard1996asymptotic} for the classical results in this direction and to \cite{jourdain2007nonlinear,bossy2011conditional,fournier2016propagation,mischler2013kac,lacker2018strong} for more recent theory. Most of the results in the literature provide non-quantitative propagation of chaos  with a few notable exceptions. In the case where the coefficients of  \eqref{eq:MVSDE} are linear in measure and globally Lipschitz continuous,  \cite{sznitman1991topics} showed that $W_2(\claw(Y^{i,N}_t), \claw(X_t))=O(N^{-1/2})$.  We refer to Sznitman's result as strong propagation of chaos. Note that, in this work, we treat the case of McKean-Vlasov SDEs with coefficients with a general  dependence in measure. In the case of Lipschitz continuous dependence in measure in the 2-Wasserstein metric, the rate of strong propagation of chaos deteriorates with the dimension $d$, \cite[Ch. 1]{carmona2016lectures}.  We demonstrate that under regularity assumptions on $b$ and $\sigma$ in terms of L-derivatives, we have  a strong error bound in fourth moment that is dimension-independent: 
 \begin{theorem}[Theorem \ref{strong error}, uniform strong propagation of chaos] 
Assume \eqref{eq:Int}. Suppose that $b, \sigma \in \cM_3( \bR^d \times \cP_2(\bR^d))$. Then
$$ \bE \Big[ W_{\cC_T,2} \big( \nlaw[], \couplaw[] \big)^4 \Big] \leq \bE \bigg[ \frac{1}{N} \sum_{i=1}^N  \bigg( \sup_{t \in [0,T]} \big| X^i_{t} - Y^{i,N}_{t} \big|^4 \bigg) \bigg]  \leq \frac{C}{N^2} , $$ 
for some constant $C>0$.
\end{theorem}

In the case of interacting diffusions (McKV-SDEs), we wish to consider the corresponding antithetic MLMC estimator to estimate $\Phi(\mclaw[T]) $. As before, we set 
 $N_{\ell}:=2^{\ell}$. For each particle system $\{ Y^{i,2N} \}_{1 \leq i \leq 2N}$, we define  two sub-particle systems to have the same number of particles, each defined by
\begin{equation*} 
\begin{split}
       & Y^{i,2N,(1)}_t  =  \xi_i + \int_0^t  b \bigg( Y^{i,2N,(1)}_r, \twonlawone[r] \bigg)  \,dr + \int_0^t  \sigma \bigg( Y^{i,2N,(1)}_r,\twonlawone[r]  \bigg)    \,  dW^i_r, \quad 1 \leq i \leq N,  \\
    & Y^{i,2N,(2)}_t  = \xi_i + \int_0^t  b \bigg( Y^{i,2N,(2)}_r, \twonlawtwo[r]  \bigg)  \,dr + \int_0^t  \sigma \bigg( Y^{i,2N,(2)}_r,\twonlawtwo[r]  \bigg)    \,  dW^i_r, \quad N+1 \leq i \leq 2N,
\end{split}
    \end{equation*}
    where
\begin{equation} \twonlawone[r]:= \frac{1}{N} \sum_{i=1}^N \delta_{Y^{i,2N,(1)}_r} \quad \quad \text{ and } \quad \quad \twonlawtwo[r]:= \frac{1}{N} \sum_{i=N+1}^{2N} \delta_{Y^{i,2N,(2)}_r}. \label{eq: def empirical measures two parts}  \end{equation} 
Unlike the i.i.d. setting  considered above, these particles are not independent. The corresponding antithetic MLMC estimator for $\Phi(\mu)$ is again given by \eqref{eq AMLMC}, where $\nlawelli{T}{\ell}{\theta}{\ell}$, $\nlawoneelli{T}{\ell}{\theta}{\ell}$ and $ \nlawtwoelli{T}{\ell}{\theta}{\ell}$ are defined similarly as $\nlawell{T}{\ell}$, $\nlawoneell{T}{\ell}$ and $\nlawtwoell{T}{\ell}$ respectively, but correspond to the copies of probability spaces indexed by $\ell \in \{0, \ldots, L \}$ and $\theta \in \{ 1, \ldots, M_{\ell} \}$. Each probability space (indexed by  $\ell$, $\theta$) supports particles with initial conditions $\xi_{i,\ell, \theta }$, $i \in \{1, \ldots, N_{\ell} \}$,  driven by Brownian motions $W^{i,\ell, \theta},$ $i \in \{1, \ldots, N_{\ell} \}$. 

 \subsubsection{Complexity analysis of A-MLMC for interacting diffusions}
 Our analysis of complexity relies heavily on the calculus on 
$(\mathcal{P}_2(\bR^d),W_2)$  and we follow the approach presented by P$.$ Lions in his course at Coll\`{e}ge de France \cite{lions2014cours} (redacted by Cardaliaguet \cite{cardaliaguet2010notes}). The important object in our study, similar
to \cite{cardaliaguet2015master}, is the PDE written on the space $[0,T]\times  \cP_2(\R^d)$, which corresponds to the lifted semigroup and comes from the It\^{o}'s formula of functionals of measures established in \cite{buckdahn2017mean} and \cite{chassagneux2014probabilistic}. This line of research has been recently explored in  \cite[Ch. 9]{kolokoltsov2010nonlinear} and \cite[Th. 2.1]{mischler2015new} to obtain results of quantitative propagation of chaos for a general family of particle systems. We shall adopt the notion of L-derivatives (see Appendix A), as well as the notion of the class $\cM_k$ of $k$th order differentiable functions in the sense of L-derivatives (see Definition \ref{eq:classmk}).

To proceed with the analysis of complexity for interacting diffusions, we first define the notions of \emph{$p$th order interactions} and \emph{order of interactions}.
\begin{definition}[Interacting kernels with $p$th order interactions] \label{def interactions complexity} 
 $b$ and $\sigma$ are said to be of $p$\emph{th order interactions} if they take the forms 
\begin{eqnarray}
  b_i(x, \mu) & =  & \int_{\bR^d} \ldots \int_{\bR^d} \overline{b_i} (x, y_1, \ldots, y_p) \, \mu(dy_1) \ldots \mu(dy_p), \label{pth order b} \\
  \sigma_{i,j}(x, \mu) & = & \int_{\bR^d} \ldots \int_{\bR^d} \overline{{\sigma}_{i,j}}(x, y_1, \ldots, y_p) \, \mu(dy_1) \ldots \mu(dy_p) , \label{pth order sigma}
 \end{eqnarray}
    where $ \overline{b_i}, \overline{{\sigma}_{i,j}} : (\bR^d)^{p+1} \to \bR$ are continuous functions, for each $i,j \in \{1, \ldots, d\}$. 
 \end{definition}    
     Hence, in any particle system with $b$ and $\sigma$ given by \eqref{pth order b} and \eqref{pth order sigma}, each particle interacts with $N^p$ other particles. Since there are $N$ particles in total, the number of interactions of the entire system is $N^{p+1}$. 
   
   For $i \in \{1, \ldots, k \}$, let $S_k := \{Y^{i,N_k}\}_{1 \leq i \leq N_k}$ denote an interacting particle system. Then the \emph{order of interactions} of an estimator composed of particle systems $S_1, \ldots, S_k$ is defined by
    $$ \text{Order of interactions of estimator} := \sum_{j=1}^k \bigg[ \text{Order of interactions of particle system } S_j \bigg].$$

 In the subsequent analysis of this section, we assume that $b$ and $\sigma$ have the forms \eqref{pth order b} and \eqref{pth order sigma} respectively. We now compare the A-MLMC estimator with the ensemble estimator studied in \cite{{chassagneux2019weak}}. The ensemble estimator corresponds to
$$ Q_{M,N}:= \frac{1}{M} \sum_{\theta=1}^M \Phi(\mu^{N,\theta}_T), $$where $\mu^{N,\theta}_T$ denotes the empirical measure of the particles obtained for each i.i.d. {cloud }$\theta \in \{1, \ldots, M\}$. If $\Phi \in \cM_3$, $\overline{b_i} \in C^{3}_{b, \text{Lip}} ( (\bR^d)^{p+1} )$ (see Section \ref{section notations} for its definition) and that $\overline{{\sigma}_{i,j}}$ belongs to $C^{3}_{b, \text{Lip}} ( (\bR^d)^{p+1} ) $ and is uniformly bounded, then  it follows by \cite{chassagneux2019weak} that the number of interactions is of the order $O(\eps^{-2-p})$. By introducing Romberg extrapolation to the ensembles of particles \cite[Sec 1.1 and Th 2.17]{chassagneux2019weak}, the number of interactions can be reduced to the order $O(\eps^{-2-p/k})$,  under the assumption that $\Phi \in \cM_{2k+1}$, $\overline{b_i} \in C^{2k+1}_{b, \text{Lip}} ( (\bR^d)^{p+1} )$ and that $\overline{{\sigma}_{i,j}}$ belongs to $C^{2k+1}_{b, \text{Lip}} ( (\bR^d)^{p+1} ) $ and is uniformly bounded.  It is proven in Theorem \ref{complexity no time discret}  that the A-MLMC estimator $ {\cA}^{\text{A-MLMC}}$ (almost) achieves an optimal order of interactions (for $p=1$), whilst only requiring  $\Phi \in \cM_{4}$, $\overline{b_i} \in C^{4}_{b, \text{Lip}} ( (\bR^d)^{p+1} )$ and $\overline{{\sigma}_{i,j}} \in C^{4}_{b, \text{Lip}} ( (\bR^d)^{p+1} ) $. The table below provides detailed comparison among the aforementioned methods.

\begin{center}
\begin{table}[ht]
\caption{Comparison of the order of interactions for different estimators}
\begin{tabular}{ |l|l|l|l|l| }
\hline
\multirow{2}{*}{\quad \quad Estimator}  & \multirow{2}{*}{\thead{ \quad \quad \quad Order of \\ \quad \quad \quad interactions} } & \multicolumn{3}{ |c| }{Regularity assumption of} \\ 
&& \quad \quad \,   $\overline{b_i}$ & \quad \quad \, \,  $\overline{{\sigma}_{i,j}}$  & \, \, $\Phi$ \\ 
\hline
 \thead{Ensembles of particles \\ {} } &  \quad \quad \,  $O(\eps^{-2-p})$ & $C^{3}_{b, \text{Lip}} ( (\bR^d)^{p+1} )$ & \thead{$C^{3}_{b, \text{Lip}} ( (\bR^d)^{p+1} )$ \\ and uniformly bounded } & \, $\cM_3$ \\
\hline   
\thead{Romberg extrapolation \\ (from \cite[Sec 1.1 and Th 2.17]{chassagneux2019weak})} & \quad \quad \,  $O(\eps^{-2-p/k})$ & $C^{2k+1}_{b, \text{Lip}} ( (\bR^d)^{p+1} )$ & \thead{$C^{2k+1}_{b, \text{Lip}} ( (\bR^d)^{p+1} )$ \\ and uniformly bounded } & $\cM_{2k+1} $ \\
\hline
\thead{Antithetic MLMC \\ (for estimator \eqref{eq AMLMC}) \\ in Theorem \ref{complexity no time discret}} & \thead{$O(\eps^{-2} (\log \eps)^2)$, for $p=1$, \\ $ O(\eps^{-1-p})$, \, \, \, \,   for $p>1$.} & $C^{4}_{b, \text{Lip}} ( (\bR^d)^{p+1} )$ & \, \,  $C^{4}_{b, \text{Lip}} ( (\bR^d)^{p+1} )$ & \, $\cM_{4} $ \\
\hline
\end{tabular} 
\end{table}
\end{center}
The following theorems are the two main results concerning the antithetic MLMC estimator for interacting particle systems. The first theorem gives an analogue of Theorem \ref{anti initial intro} in the case of interacting particle systems. 
\begin{theorem}[Theorem \ref{variance antithetic non-discretised}, variance of antithetic difference] 
Assume \eqref{eq:Int}. Suppose that $b, \sigma \in \cM_4 \big( \bR^d  \times \mathcal{P}_2(\bR^d) \big) $ and $\Phi \in \cM_4 \big(  \mathcal{P}_2(\bR^d) \big) $.  Then 
$$\text{Var} \Big[ \Phi (\twonlaw[T]) -  \frac{1}{2} \big( \Phi (\twonlawone[T]) + \Phi (\twonlawtwo[T]) \big) \Big] \leq   \frac{C}{N^2}, $$ 
where $C$ is a constant that depends on $\Phi$, $b$, $\sigma$ and $T$, but does not depend on $N$.
\end{theorem}
The following theorem gives a quantitative estimate on the order of interactions of the antithetic MLMC estimator. 
\begin{theorem}[Theorem \ref{complexity no time discret}]  
Assume \eqref{eq:Int}. Suppose that $b$ and $\sigma$ are of the forms \eqref{pth order b} and \eqref{pth order sigma} respectively. Furthermore, suppose that $b, \sigma \in \cM_4 \big( \bR^d  \times \mathcal{P}_2(\bR^d) \big) $ and $\Phi \in \cM_4 \big(  \mathcal{P}_2(\bR^d) \big) $.  Then there exist constants $C_1, C_2>0$ such that for any $\eps < e^{-1}$, there exist a value $L$ and a sequence $\{ M_{\ell} \}_{\ell=0}^L$ such that the mean-square error of $ {\cA}^{\text{A- MLMC}}$ $($given by \eqref{eq AMLMC}$)$ is bounded by
$$  \bE \big[ \big(  {\cA}^{\text{A- MLMC}} - \Phi(\mclaw[T]) \big)^2 \big]  \leq C_1 \eps^2 $$ 
and the order of interactions of $ {\cA}^{\text{A- MLMC}}$  is bounded by 
\[ \text{\emph{Order of interactions}} \, \big(  {\cA}^{\text{A- MLMC}} \big) \leq        \begin{cases} 
      C_2 \eps^{-2} (\log \eps)^2, & p=1, \\
      C_2 \eps^{-1-p}, & p>1.
   \end{cases}
\]
\end{theorem}
For practical purposes, time discretisation is generally needed to simulate SDEs. We consider the time discretisation of \eqref{eq:MVSDE}, as  in seminal papers by Bossy and Talay \cite{bossy1996convergence,bossy1997stochastic}, by working with an Euler scheme. Take partition $\{t_k\}_k$ of $[0,T]$, with $t_k-t_{k-1}=h$ and define $\eta{(t)} := t_k\, \text{if}~t\in[t_k,t_{k+1})$. The continuous Euler scheme reads 
\begin{equation}
Z^{i,N,h}_{t}  = Z^{i,N,h}_{t_{k}} +  b(Z^{i,N,h}_{\eta{(t)}},\mu^{Z,N,h}_{\eta{(t)}} ) (t-t_k)
+ \sigma(Z^{i,N,h}_{\eta{(t)}},\mu^{Z,N,h}_{\eta{(t)}} )  (  W^i_{t} - W^i_{{t_{k}}}). \label{eq:Euler}
\end{equation} 
Section 6 extends the antithetic MLMC estimator to include time discretisation (along with its complexity analysis), so as to be implementable. The numerical simulations of the algorithm in Section 6 can be found in \cite{haji2018multilevel}.

\subsection{Outline of the paper} \label{section notations} 

Here is an outline of the main  results of the article. 

Section 3 establishes the result on the complexity of A-MLMC in Section \ref{complexity iid subsection} by proving Theorem \ref{antithetic initial separation}, which concerns the antithetic difference in the i.i.d. case. Theorem \ref{antithetic initial separation} generalises the result in \cite{giles2015multilevel} (Section 9) from functionals in measure of the form \eqref{eq: giles antithetic special form} to general functionals in measure, which ultimately allows us to prove the complexity result:  Theorem \ref{thm A mlmc complexity}. 

Subsequently, in Section 4, Theorem \ref{strong error} proves a dimension-independent rate of uniform strong propagation of chaos for sufficiently smooth drift and diffusion functions. This is a considerable generalisation from \cite{sznitman1991topics}, which assumes the drift and diffusion functions to be linear in measure.

In Section 5, we show from Theorem \ref{complexity no time discret}  that, under sufficient regular conditions on $b$ and $\sigma$ (having forms \eqref{pth order b} and \eqref{pth order sigma} respectively), the order of interactions of $ {\cA}^{\text{A- MLMC}}$  (given by \eqref{eq AMLMC}) is bounded by $C\eps^{-2} (\log \eps)^2$, for $p=1$; $C \eps^{-1-p}$, for $p >1$. 

Finally, in Section 6, we apply an Euler time-discretisation to the A-MLMC method. In Theorem \ref{MLMC time discret complexity}, we show that, under sufficient regular conditions on $b$ and $\sigma$ (having forms \eqref{pth order b} and \eqref{pth order sigma} respectively), the computational complexity of estimator \eqref{eq: MLMC def }  is bounded by $ C \eps^{-2-p}$, where $p \geq 1$, which is still a considerable improvement compared to direct Monte-Carlo simulation. 

Since this work relies heavily on the theory of differentiation in measure developed by P$.$ Lions in his course at Coll\`{e}ge de France \cite{lions2014cours}, the reader is directed to Appendices \ref{ linear functional derivatives and L-derivatives appendix}  and \ref{appendix weak error analysis}  for further details.

\section{A-MLMC for i.i.d. random variables} \label{sec iid}
The main result in this section is Theorem \ref{antithetic initial separation}, from which we can prove  Theorem \ref{thm A mlmc complexity}. We begin this section with the following lemma on the $W_2$ metric.
\begin{lemma} \label{W2 lemma} 
Let $\eta \in \bR^d $ and $m \in \cP_2(\bR^d)$. Then
$$ W_2 \Big( \frac{1}{N} \delta_{\eta} + \frac{N-1}{N} m,m \Big)^2 \leq \frac{2}{N} \bigg( |\eta|^2 + \int_{\bR^d} |x|^2 \, m(dx) \bigg).$$ 
\end{lemma}
\begin{proof}
Let $Y$ be a random variable with law $m$ and let $\Omega' \in \cF$ be a measurable event that is independent of $\sigma(Y)$, with probability $\frac{N-1}{N}.$ Let $X$ be a random variable defined by
\[ X(\omega) := \begin{cases} 
      Y(\omega),  & \omega \in \Omega', \\
           \eta,  & \omega \not\in \Omega'.
   \end{cases}
\]
Then the law of $X$ is $\frac{1}{N} \delta_{\eta} + \frac{N-1}{N} m$. Therefore, by the definition of the 2-Wasserstein metric,
\begin{eqnarray}
 W_2 \Big( \frac{1}{N} \delta_{\eta} + \frac{N-1}{N} m,m \Big)^2  & \leq & \bE \big[ |X-Y|^2 \big] \nonumber \\
 & = & \bE \big[ |X-Y|^2 \big| \Omega' \big] \bP(\Omega') + \bE \big[ |X-Y|^2 \big| (\Omega')^{c} \big] \bP((\Omega')^{c} ) \nonumber \\
 & = & \frac{1}{N} \bE[ |\eta - Y|^2] \nonumber \\
 & \leq & \frac{2}{N} ( |\eta|^2 + \bE[|Y|^2]). \nonumber 
\end{eqnarray}
\end{proof}
For any functional from $\cP_2(\bR^d)$ to $\bR$, the following lemma gives a bound on the error between the value of empirical measures under the functional and its limiting law under the functional. It relies on the regularity conditions stipulated in Proposition \ref{theorem lions linear functional}.  The proof of the following lemma is similar to Lemma 5.10 in \cite{delarue2019master}. However, the following result is slightly different in terms of hypotheses, as the first and second order linear functional derivatives are only of linear and quadratic growth respectively (Proposition \ref{theorem lions linear functional}), whereas they are assumed to be uniformly bounded and $W_1$-Lipschitz continuous in Lemma 5.10 of \cite{delarue2019master}. In return, we require a much higher moment (12 v.s. 4 in Lemma 5.10 of \cite{delarue2019master}). The following result is stated in a way with a constant that does not depend on the functional of measure, nor on the limiting law, so that it is useful with the relevant conditioning argument in the proof of Proposition \ref{important proposition lemma strong error}. The technique of the following proof is also adopted in the proof of Theorem \ref{antithetic initial separation}.
\begin{lemma} \label{delarue theorem modified} 
Let $U \in \cM_3(\cP_2(\bR^d))$. Let $m_0 \in \cP_{12}(\bR^d)$ and $m^N= \frac{1}{N} \sum_{i=1}^N \delta_{{\zeta}_i}$, where ${\zeta}_1, \ldots, {\zeta}_N$ are i.i.d samples with law $m_0$. Then there exists a constant $C>0$ (which does not depend on $U$, ${\zeta}_1, \ldots, {\zeta}_N$ and $m_0$) such that  
$$ \bE \big[ \big| U( m^N)- U(m_0) \big|^4 \big] \leq \frac{C}{N^2} \prod_{i=1}^3 \Big( 1+ \| \partial^i_{\mu} U \|^4_{\infty}  \Big) \bigg( 1+ \int_{\bR^d} |x|^{12} \, m_0 (dx) \bigg) .$$ 
\end{lemma}
\begin{proof}
In this proof, $C$ denotes an absolute constant that does not depend on $U$, ${\zeta}_1, \ldots, {\zeta}_N$ and $m_0$, whose value may vary from line to line. By the definition of linear functional derivatives, we have
\begin{eqnarray}
U( m^N)- U(m_0) & = & \int_0^1 \int_{\bR^d}   \frac{\delta U}{\delta m}(\lambda m^N + (1- \lambda) m_0,v) \, (m^N-m_0) (dv) \, d \lambda \nonumber \\
& = & \frac{1}{N} \sum_{i=1}^N \int_0^1 \varphi^i_{\lambda} \, d \lambda, \nonumber 
\end{eqnarray}
where, for $i \in \{1, \ldots, N\}$ and $\lambda \in [0,1]$,
\begin{equation}
\varphi^i_{\lambda}  =   \frac{\delta U}{\delta m}(\lambda m^N + (1- \lambda) m_0,{\zeta}_i)  -   \widetilde{\bE} \bigg[ \frac{\delta U}{\delta m}(\lambda m^N + (1- \lambda) m_0,\tilde{{\zeta}}) \bigg]. \label{eq: varphi i ell def} 
\end{equation}
By the bound on $\frac{\delta U}{\delta m}$ in Proposition \ref{theorem lions linear functional}, we know that for distinct $i, j \in \{1, \ldots, N\},$
\begin{equation} \bE \big[ (\varphi^i_{\lambda})^4 + (\varphi^i_{\lambda})^2 (\varphi^j_{\lambda})^2 + \varphi^i_{\lambda} (\varphi^j_{\lambda})^3 \big] \leq C  \| \pmu U \|^4_{\infty} \bE[|\zeta_1|^4].   \label{ square integ phi } \end{equation}
We have the estimate 
\begin{eqnarray} \bE \big[ \big| U( m^N)- U(m_0) \big|^4 \big] & \leq & \frac{1}{N^4} \int_0^1  \bE \bigg[ \bigg(  \sum_{i=1}^N \varphi^i_{\lambda} \bigg)^4 \bigg] \, d \lambda \nonumber \\
& \leq &  C \bigg( \frac{1}{N^2} \| \pmu U \|^4_{\infty} \bE[|\zeta_1|^4] \nonumber \\
&& + \frac{1}{N^4} \int_0^1 \bE \bigg[   \sum_{i_1, i_2, i_3 \text{ distinct }}
\varphi^{i_1}_{\lambda} \varphi^{i_2}_{\lambda}
(\varphi^{i_3}_{\lambda})^2 + \sum_{\substack{i_1, i_2, i_3, i_4 \\ \text{ distinct }}}  \varphi^{i_1}_{\lambda} \varphi^{i_2}_{\lambda} \varphi^{i_3}_{\lambda} \varphi^{i_4}_{\lambda} \bigg] \, d \lambda \bigg) . \nonumber \\
&& \label{eq: S2 overall bound} \end{eqnarray}
For any distinct $i_1, i_2, i_3$, we define $m^{N,-(i_1,i_2,i_3)}:= \frac{1}{N-3} \sum_{\ell \neq i_1,i_2,i_3} \delta_{{\zeta}_{\ell}}$, which implies that 
$$ m^N- m^{N,-(i_1,i_2,i_3)}= \frac{1}{N} (\delta_{{\zeta}_{i_1}} + \delta_{{\zeta}_{i_2}} + \delta_{{\zeta}_{i_3}})  - \frac{3}{N(N-3)} \sum_{\ell \neq i_1, i_2,i_3} \delta_{{\zeta}_{\ell}} .$$ 
By the definition of second-order linear functional derivatives, we observe that
\begin{eqnarray}
& & \frac{\delta U}{\delta m}(\lambda m^N + (1- \lambda) m_0,{\zeta}_i) - \frac{\delta U}{\delta m}(\lambda m^{N,-(i_1,i_2,i_3)} + (1- \lambda) m_0,{\zeta}_i) \nonumber \\
& = & \int_0^1 \int_{\bR^d} \frac{\delta^2 U}{\delta m^2} \Big( s \lambda m^N + (1-s) \lambda m^{N,-(i_1,i_2,i_3)} + (1- \lambda) m_0 , {\zeta}_i, v \Big) \, (m^N- m^{N,-(i_1,i_2,i_3)})(dv) \, ds \nonumber \\
& = & \int_0^1 \frac{1}{N} \bigg[ \sum_{\ell=i_1, i_2,i_3} \frac{\delta^2 U}{\delta m^2} \Big( s \lambda m^N + (1-s) \lambda m^{N,-(i_1,i_2,i_3)} + (1- \lambda) m_0 , {\zeta}_i, {\zeta}_{\ell} \Big) \nonumber \\
&& - \frac{3}{N-3} \sum_{\ell \neq i_1, i_2,i_3} \frac{\delta^2 U}{\delta m^2} \Big( s \lambda m^N + (1-s) \lambda m^{N,-(i_1,i_2,i_3)} + (1- \lambda) m_0 , {\zeta}_i, {\zeta}_{\ell} \Big) \bigg] \, ds. \label{eq: i3 expansion} 
\end{eqnarray} 
By the bound on $\frac{\delta^2 U}{\delta m^2}$ in Proposition \ref{theorem lions linear functional},
\begin{eqnarray}
&& \bE \bigg| \frac{\delta U}{\delta m}(\lambda m^N + (1- \lambda) m_0,{\zeta}_i) - \frac{\delta U}{\delta m}(\lambda m^{N,-(i_1,i_2,i_3)} + (1- \lambda) m_0,{\zeta}_i)  \bigg|^4 \leq \frac{C}{N^4} \| \ptwomu U \|^4_{\infty} \bE[|\zeta_1|^8] .\nonumber  \end{eqnarray} 
Similarly, by applying the same argument to the second term in \eqref{eq: varphi i ell def}, we obtain that
\begin{eqnarray}
&&  \bE \bigg| \widetilde{\bE} \bigg[ \frac{\delta U}{\delta m}(\lambda m^N + (1- \lambda) m_0,\tilde{{\zeta}}) \bigg] - \widetilde{\bE} \bigg[ \frac{\delta U}{\delta m}(\lambda m^{N,-(i_1,i_2,i_3)} + (1- \lambda) m_0,\tilde{{\zeta}}) \bigg] \bigg|^4 \leq \frac{C}{N^4} \| \ptwomu U \|^4_{\infty} \bE[|\zeta_1|^8] ,  \nonumber 
\end{eqnarray} 
which implies that
\begin{equation}  \bE | \varphi^i_{\lambda} - \varphi^{i,-(i_1, i_2,i_3)}_{\lambda} |^4 \leq \frac{C}{N^4} \| \ptwomu U \|^4_{\infty} \bE[|\zeta_1|^8], \label{eq: difference phi i lambda bound} \end{equation} 
where
\begin{equation}
\varphi^{i,-(i_1, i_2,i_3)}_{\lambda}  =  \frac{\delta U}{\delta m}(\lambda m^{N,-(i_1,i_2,i_3)} + (1- \lambda) m_0,{\zeta}_i) - \widetilde{\bE} \bigg[ \frac{\delta U}{\delta m}(\lambda m^{N,-(i_1,i_2,i_3)} + (1- \lambda) m_0,\tilde{{\zeta}})  \bigg]. \label{eq: removing three particles} 
\end{equation}
Finally, by writing $ \varphi^i_{\lambda} = ( \varphi^i_{\lambda} -  \varphi^{i,-(i_1, i_2,i_3)}_{\lambda} ) + \varphi^{i,-(i_1, i_2,i_3)}_{\lambda}  $ and applying the generalised H\"{o}lder's inequality  to \eqref{ square integ phi } and \eqref{eq: difference phi i lambda bound}, 
\begin{eqnarray}
&& \sum_{i_1, i_2, i_3 \text{ distinct }} \bE \Big[ 
\varphi^{i_1}_{\lambda} \varphi^{i_2}_{\lambda}
(\varphi^{i_3}_{\lambda})^2 \Big] \nonumber \\
& \leq &  \sum_{i_1, i_2, i_3 \text{ distinct }} \bigg[  \frac{C}{N} (1+  \| \pmu U \|^4_{\infty}) (1+  \| \ptwomu U \|^4_{\infty}) \bE[|\zeta_1|^8] + \bE \big[ \varphi^{i_1,-(i_1, i_2,i_3)}_{\lambda} \varphi^{i_2,-(i_1, i_2,i_3)}_{\lambda} \big( \varphi^{i_3,-(i_1, i_2,i_3)}_{\lambda} \big)^2 \big]  \bigg] \nonumber \\
& \leq & CN^2 (1+  \| \pmu U \|^4_{\infty}) (1+  \| \ptwomu U \|^4_{\infty}) \bE[|\zeta_1|^8] + \sum_{\substack{i_1, i_2, i_3 \\ \text{ distinct }}} \bE \big[ \varphi^{i_1,-(i_1, i_2,i_3)}_{\lambda} \varphi^{i_2,-(i_1, i_2,i_3)}_{\lambda} \big( \varphi^{i_3,-(i_1, i_2,i_3)}_{\lambda} \big)^2 \big].  \nonumber \\
&& \label{eq: three terms fourth order holder's argument} 
\end{eqnarray} 
Let $\cF^{-i}$ be the $\sigma$-algebra generated by $\zeta_1, \ldots, \zeta_N$ except $\zeta_i$. Since ${\zeta}_1, \ldots, {\zeta}_N$ are independent, for any distinct $i_1, i_2,i_3$, 
\begin{equation}
   \bE \big[ \varphi^{i_1,-(i_1, i_2,i_3)}_{\lambda} \varphi^{i_2,-(i_1, i_2,i_3)}_{\lambda} \big( \varphi^{i_3,-(i_1, i_2,i_3)}_{\lambda} \big)^2 \big] =    \bE \big[ \varphi^{i_2,-(i_1, i_2,i_3)}_{\lambda} \big( \varphi^{i_3,-(i_1, i_2,i_3)}_{\lambda} \big)^2  \bE \big[  \varphi^{i_1,-(i_1, i_2,i_3)}_{\lambda} \big| \cF^{-i_1} \big] \big] =0, \label{eq:indep conditioning} 
\end{equation}
which implies that
\begin{equation} \sum_{i_1, i_2, i_3 \text{ distinct }} \bE \Big[ 
\varphi^{i_1}_{\lambda} \varphi^{i_2}_{\lambda}
(\varphi^{i_3}_{\lambda})^2 \Big] \leq  CN^2 (1+  \| \pmu U \|^4_{\infty}) (1+  \| \ptwomu U \|^4_{\infty}) \bE[|\zeta_1|^8]. \label{i3 overall bound} \end{equation} 
Next, we define analogously the notation $\varphi^{i,-(i_1, i_2,i_3,i_4)}$ as \eqref{eq: removing three particles}. As above,
by applying the generalised H\"{o}lder's inequality  to \eqref{ square integ phi } and \eqref{eq: difference phi i lambda bound}, followed by a similar reasoning as \eqref{eq:indep conditioning}, we have 
\begin{eqnarray}
&& \sum_{\substack{i_1, i_2, i_3,i_4 \\ \text{ distinct }}} \bE \Big[ 
\varphi^{i_1}_{\lambda} \varphi^{i_2}_{\lambda} \varphi^{i_3}_{\lambda} \varphi^{i_4}_{\lambda}
 \Big] \nonumber \\
& \leq &  \sum_{\substack{i_1, i_2, i_3,i_4 \\ \text{ distinct }}} \Bigg[  \frac{C}{N^2} (1+  \| \pmu U \|^4_{\infty}) (1+  \| \ptwomu U \|^4_{\infty}) \bE[|\zeta_1|^8] + \bE \bigg[ \sum_{j=1}^4 \big( \varphi^{i_j}_{\lambda} - \varphi^{i_j,-(i_1, i_2,i_3,i_4)}_{\lambda}) \prod_{\substack{k=1 \\ k \neq j}}^4 \varphi^{i_k,-(i_1, i_2,i_3,i_4)}_{\lambda} \bigg] \nonumber \\
&& + \bE \big[ \varphi^{i_1,-(i_1, i_2,i_3,i_4)}_{\lambda} \varphi^{i_2,-(i_1, i_2,i_3,i_4)}_{\lambda}  \varphi^{i_3,-(i_1, i_2,i_3,i_4)}_{\lambda} \varphi^{i_4,-(i_1, i_2,i_3,i_4)}_{\lambda} \big] \Bigg] \nonumber \\
& \leq & CN^2 (1+  \| \pmu U \|^4_{\infty}) (1+  \| \ptwomu U \|^4_{\infty}) \bE[|\zeta_1|^8] + \sum_{\substack{i_1, i_2, i_3,i_4 \\ \text{ distinct }}} \bE \bigg[\sum_{j=1}^4 \big( \varphi^{i_j}_{\lambda} - \varphi^{i_j,-(i_1, i_2,i_3,i_4)}_{\lambda}) \prod_{\substack{k=1 \\ k \neq j}}^4 \varphi^{i_k,-(i_1, i_2,i_3,i_4)}_{\lambda}  \bigg].  \nonumber \\
&& \label{eq: four terms fourth order holder's argument} 
\end{eqnarray} 
Note that \eqref{eq: difference phi i lambda bound} only gives a growth in the order of $O(N^3)$ for the final term in \eqref{eq: four terms fourth order holder's argument}, therefore it is insufficient. 

By \eqref{eq: i3 expansion} followed by an application of the definition of third order linear functional derivatives, we have
\begin{eqnarray}
&&  \frac{\delta U}{\delta m}(\lambda m^N + (1- \lambda) m_0,{\zeta}_i) -   \frac{\delta U}{\delta m}(\lambda m^{N,-(i_1,i_2,i_3,i_4)} + (1- \lambda) m_0,{\zeta}_i) \nonumber \\
& = &  \frac{1}{N} \bigg[ \sum_{\ell=i_1, i_2,i_3,i_4} \frac{\delta^2 U}{\delta m^2} \Big(  \lambda m^{N,-(i_1,i_2,i_3,i_4)} + (1- \lambda) m_0 , {\zeta}_i, {\zeta}_{\ell} \Big) \nonumber \\
&& - \frac{4}{N-4} \sum_{\ell \neq i_1, i_2,i_3,i_4} \frac{\delta^2 U}{\delta m^2} \Big(  \lambda m^{N,-(i_1,i_2,i_3,i_4)} + (1- \lambda) m_0 , {\zeta}_i, {\zeta}_{\ell} \Big) \bigg]  + \varepsilon^{i,-(i_1, i_2, i_3, i_4)}_N, \label{ fourth order varepsilon third order linear functional derivative}
\end{eqnarray} 
where
\begin{eqnarray}
&& \varepsilon^{i,-(i_1, i_2, i_3, i_4)}_N \nonumber \\
& = & \int_0^1 \frac{s \lambda}{N^2} \Bigg[ \sum_{\ell=i_1, i_2,i_3,i_4}  \int_0^1  \bigg[ \sum_{\ell'=i_1, i_2,i_3,i_4} \frac{\delta^3 U}{\delta m^3} \Big( ts  \lambda m^N + (1-ts) \lambda  m^{N,-(i_1,i_2,i_3,i_4)} + (1- \lambda) m_0 , \nonumber \\
&& {\zeta}_i, {\zeta}_{\ell}, {\zeta}_{\ell'} \Big) - \frac{4}{N-4} \sum_{\ell' \neq i_1, i_2,i_3,i_4} \frac{\delta^3 U}{\delta m^3} \Big( ts  \lambda m^N + (1-ts) \lambda  m^{N,-(i_1,i_2,i_3,i_4)} + (1- \lambda) m_0 ,  {\zeta}_i, {\zeta}_{\ell}, {\zeta}_{\ell'} \Big) \bigg] \, dt   \nonumber \\
&& - \frac{4}{N-4} \sum_{\ell \neq i_1, i_2,i_3,i_4}  \int_0^1 \bigg[  \sum_{\ell'=i_1, i_2,i_3,i_4} \frac{\delta^3 U}{\delta m^3} \Big( ts  \lambda m^N + (1-ts) \lambda  m^{N,-(i_1,i_2,i_3,i_4)} + (1- \lambda) m_0 , \nonumber \\
&& {\zeta}_i, {\zeta}_{\ell}, {\zeta}_{\ell'} \Big)  - \frac{4}{N-4} \sum_{\ell' \neq i_1, i_2,i_3,i_4} \frac{\delta^3 U}{\delta m^3} \Big( ts  \lambda m^N + (1-ts) \lambda  m^{N,-(i_1,i_2,i_3,i_4)} + (1- \lambda) m_0 ,  {\zeta}_i, {\zeta}_{\ell}, {\zeta}_{\ell'} \Big) \bigg] \, dt \Bigg] \,ds,  \nonumber 
\end{eqnarray}
which implies that
$$ \bE|\varepsilon^{i,-(i_1, i_2, i_3, i_4)}_N|^4 \leq  \frac{C}{N^8} \| \pthreemu U \|^4_{\infty} \bE[|\zeta_1|^{12}], $$
by the bound on $\frac{\delta^3 U}{\delta m^3}$ in Proposition \ref{theorem lions linear functional}. Repeating the same argument to the other term in \eqref{eq: varphi i ell def} gives
\begin{eqnarray}
&& \varphi^{i}_{\lambda} - \varphi^{i,-(i_1, i_2,i_3,i_4)}_{\lambda} \nonumber \\
& = &  \int_{\bR^d} \frac{\delta^2 U}{\delta m^2} \Big(  \lambda m^{N,-(i_1,i_2,i_3,i_4)} + (1- \lambda) m_0 , {\zeta}_i, v \Big) \, (m^N- m^{N,-(i_1,i_2,i_3,i_4)})(dv)  \nonumber \\
&& - \tilde{\bE} \bigg[  \int_{\bR^d} \frac{\delta^2 U}{\delta m^2} \Big(  \lambda m^{N,-(i_1,i_2,i_3,i_4)} + (1- \lambda) m_0 , \tilde{\zeta}, v \Big) \, (m^N- m^{N,-(i_1,i_2,i_3,i_4)})(dv)  \bigg] + \tilde{\varepsilon}^{i,-(i_1, i_2, i_3, i_4)}_N, \nonumber 
\end{eqnarray}
where
\begin{equation} \bE|\tilde{\varepsilon}^{i,-(i_1, i_2, i_3, i_4)}_N|^4 \leq  \frac{C}{N^8} \| \pthreemu U \|^4_{\infty} \bE[|\zeta_1|^{12}].  \label{eps fourth moment bound}  \end{equation}
Note that we can write the difference $ \varphi^{i_1}_{\lambda} - \varphi^{i_1,-(i_1, i_2,i_3,i_4)}_{\lambda} - \tilde{\varepsilon}^{i_1,-(i_1, i_2, i_3, i_4)}_N $ as
$$ \varphi^{i_1}_{\lambda} - \varphi^{i_1,-(i_1, i_2,i_3,i_4)}_{\lambda} - \tilde{\varepsilon}^{i_1,-(i_1, i_2, i_3, i_4)}_N = \sum_{j=2}^4 F_j(\big( \zeta_r \big)_{r \neq i_1, \ldots, i_4}, \zeta_{i_1}, \zeta_{i_j}),$$
for some measurable functions $F_2, F_3, F_4:(\bR^d)^{N-2} \to \bR$. Therefore,
\begin{eqnarray}
&& \bE \bigg[\Big( \varphi^{i_1}_{\lambda} - \varphi^{i_1,-(i_1, i_2,i_3,i_4)}_{\lambda} - \tilde{\varepsilon}^{i_1,-(i_1, i_2, i_3, i_4)}_N \Big) \varphi^{i_2,-(i_1, i_2,i_3,i_4)}_{\lambda} \varphi^{i_3,-(i_1, i_2,i_3,i_4)}_{\lambda} \varphi^{i_4,-(i_1, i_2,i_3,i_4)}_{\lambda} \bigg] \nonumber \\
& = & \bE \bigg[\Big(\sum_{j \in \{3, 4 \}} F_j(\big( \zeta_r \big)_{r \neq i_1, \ldots, i_4}, \zeta_{i_1}, \zeta_{i_j}) \Big)  \varphi^{i_3,-(i_1, i_2,i_3,i_4)}_{\lambda} \varphi^{i_4,-(i_1, i_2,i_3,i_4)}_{\lambda} \bE \Big[ \varphi^{i_2,-(i_1, i_2,i_3,i_4)}_{\lambda} \Big| \cF^{-i_2} \Big] \bigg] \nonumber \\
& & + \bE \bigg[ F_2(\big( \zeta_r \big)_{r \neq i_1, \ldots, i_4}, \zeta_{i_1}, \zeta_{i_2}) \varphi^{i_2,-(i_1, i_2,i_3,i_4)}_{\lambda} \varphi^{i_4,-(i_1, i_2,i_3,i_4)}_{\lambda}  \bE \Big[ \varphi^{i_3,-(i_1, i_2,i_3,i_4)}_{\lambda} \Big| \cF^{-i_3} \Big] \bigg]=0. \nonumber 
\end{eqnarray}
Applying the generalised H\"{o}lder's inequality to \eqref{eps fourth moment bound} and \eqref{ square integ phi } gives
\begin{eqnarray} && \bE \bigg[\Big( \varphi^{i_1}_{\lambda} - \varphi^{i_1,-(i_1, i_2,i_3,i_4)}_{\lambda}\Big) \varphi^{i_2,-(i_1, i_2,i_3,i_4)}_{\lambda} \varphi^{i_3,-(i_1, i_2,i_3,i_4)}_{\lambda} \varphi^{i_4,-(i_1, i_2,i_3,i_4)}_{\lambda} \bigg] \nonumber \\
& \leq & \frac{C}{N^2} \| \pthreemu U \|_{\infty} \big( \bE[|\zeta_1|^{12}] \big)^{1/4}  \| \pmu U \|^3_{\infty} \big( \bE[|\zeta_1|^{4}] \big)^{3/4} \leq  \frac{C}{N^2} \Big( 1+\| \pmu U \|^4_{\infty} \Big) \Big( 1+\| \pthreemu U \|^4_{\infty} \Big) (1+ \bE[|\zeta_1|^{12}]). \nonumber  \end{eqnarray}
By the same reasoning, we can show that
\begin{equation}
\sum_{\substack{i_1, i_2, i_3,i_4 \\ \text{ distinct }}} \bE \bigg[\sum_{j=1}^4 \big( \varphi^{i_j}_{\lambda} - \varphi^{i_j,-(i_1, i_2,i_3,i_4)}_{\lambda}) \prod_{\substack{k=1 \\ k \neq j}}^4 \varphi^{i_k,-(i_1, i_2,i_3,i_4)}_{\lambda}  \bigg]  \leq C N^2 \Big( 1+\| \pmu U \|^4_{\infty} \Big) \Big( 1+\| \pthreemu U \|^4_{\infty} \Big) (1+ \bE[|\zeta_1|^{12}]). \label{eq: i4 final bound}
\end{equation}
We conclude the result by combining  \eqref{eq: S2 overall bound}, \eqref{i3 overall bound}, \eqref{eq: four terms fourth order holder's argument} and \eqref{eq: i4 final bound}.
\end{proof}


\begin{theorem}[Antithetic error on i.i.d. random variables]  \label{antithetic initial separation}  Suppose that $\mu \in \cP_8(\bR^d)$ and $ U \in \cM^L_4(\cP_2(\bR^d))$.
Then there exists a constant $C>0$ such that 
$$ \bE \big| U (\twonlaw[]) -  \frac{1}{2} \big( U (\twonlawone[]) + U (\twonlawtwo[]) \big) \big|^2 \leq \frac{C}{N^2}. $$
\end{theorem} 

\begin{proof}[Proof of theorem \ref{antithetic initial separation}]
For simplicity of notations, let 
$$\mu_{2N}:= \twonlaw[], \quad \mu_{2N,(1)}:= \twonlawone[], \quad \mu_{2N,(2)}:= \twonlawtwo[].$$
For every $t \in [0,1],$ let
$$ m^{2N}_t:= (1-t) \mu + t \mu_{2N}, \quad m^{2N,(1)}_t:= (1-t) \mu + t \mu_{2N,(1)}, \quad  m^{2N,(2)}_t:= (1-t) \mu + t \mu_{2N,(2)}. $$ 
We define 
\begin{align*}
[0,1] \ni t \mapsto f(t) = U \big((1-t)\mu + t \mu_{2N} \big) =  U \big(\mu + t (\mu_{2N}-\mu) \big) \in \bR
\end{align*}
and apply Taylor-Lagrange formula to $f$ up to order $2$, namely
\begin{align*}
f(1)-f(0) 
&=  f^{'}(0) + \int_0^1(1-t) f^{(2)}(t) \, d t.
\end{align*}
This yields
\begin{align}
U(\mu_{2N}) - U(\mu) = 
{\int_{\bR^{d}}\frac{\delta U}{\delta m}(\mu)(\mathbf{y})\,(\mu_{2N}-\mu) (d\mathbf{y})}
+  \int_0^1(1-t) \bigg[ {\int_{\bR^{2d}}\frac{\delta^{2} U}{\delta m^{2}}(m^{2N}_t)(\mathbf{y})\,(\mu_{2N}-\mu)^{\otimes 2} (d \mathbf{y})}  \bigg] \, dt. \label{antithetic taylor 2N}
\end{align}
Similarly, 
\begin{eqnarray}
U(\mu_{2N,(1)}) - U(\mu) & = & {\int_{\bR^{d}}\frac{\delta U}{\delta m}(\mu)(\mathbf{y})\,(\mu_{2N,(1)}-\mu) (d\mathbf{y})} \nonumber \\
&& +  \int_0^1(1-t) \bigg[ {\int_{\bR^{2d}}\frac{\delta^{2} U}{\delta m^{2}}(m^{2N,(1)}_t)(\mathbf{y})\,(\mu_{2N,(1)}-\mu)^{\otimes 2} (d \mathbf{y})}  \bigg] \, dt \label{antithetic taylor 2N,(1)}
\end{eqnarray}
and
\begin{eqnarray}
U(\mu_{2N,(2)}) - U(\mu) & = & {\int_{\bR^{d}}\frac{\delta U}{\delta m}(\mu)(\mathbf{y})\,(\mu_{2N,(2)}-\mu) (d\mathbf{y})} \nonumber \\
&& +  \int_0^1(1-t) \bigg[ {\int_{\bR^{2d}}\frac{\delta^{2} U}{\delta m^{2}}(m^{2N,(2)}_t)(\mathbf{y})\,(\mu_{2N,(2)}-\mu)^{\otimes 2} (d \mathbf{y})}  \bigg] \, dt. \label{antithetic taylor 2N,(2)}
\end{eqnarray}
Computing the difference of \eqref{antithetic taylor 2N} with the arithmetic average of \eqref{antithetic taylor 2N,(1)} and \eqref{antithetic taylor 2N,(2)} gives
\begin{eqnarray}
U (\mu_{2N}) -  \frac{1}{2} \big( U (\mu_{2N,(1)}) + U (\mu_{2N,(2)}) \big) & = & \int_0^1(1-t) \bigg[ {\int_{\bR^{2d}}\frac{\delta^{2} U}{\delta m^{2}}(m^{2N}_t)(\mathbf{y})\,(\mu_{2N}-\mu)^{\otimes 2} (d \mathbf{y})}  \bigg] \, dt \nonumber \\
&& - \frac{1}{2} \int_0^1(1-t) \bigg[ {\int_{\bR^{2d}}\frac{\delta^{2} U}{\delta m^{2}}(m^{2N,(1)}_t)(\mathbf{y})\,(\mu_{2N,(1)}-\mu)^{\otimes 2} (d \mathbf{y})}  \bigg] \, dt \nonumber \\
&& - \frac{1}{2} \int_0^1(1-t) \bigg[ {\int_{\bR^{2d}}\frac{\delta^{2} U}{\delta m^{2}}(m^{2N,(2)}_t)(\mathbf{y})\,(\mu_{2N,(2)}-\mu)^{\otimes 2} (d \mathbf{y})}  \bigg] \, dt. \nonumber \\
&& \label{eq antithetic second order derivatives expansion} 
\end{eqnarray}
The rest of the proof is very similar to the proof of Lemma \ref{delarue theorem modified}. It suffices to consider only the first term in \eqref{eq antithetic second order derivatives expansion}. The other two terms can be handled in a similar way. We rewrite
\begin{eqnarray}
  &&   {\int_{\bR^{2d}}\frac{\delta^{2} U}{\delta m^{2}}(m^{2N}_t)(\mathbf{y})\,(\mu_{2N}-\mu)^{\otimes 2} (d \mathbf{y})} \nonumber \\ 
  & = & {\int_{\bR^{d}} \bigg[ \frac{1}{2N} \sum_{i=1}^{2N} \frac{\delta^{2} U}{\delta m^{2}}(m^{2N}_t)(\xi_i,y_2) - \int_{\bR^d} \frac{\delta^{2} U}{\delta m^{2}}(m^{2N}_t)(z,y_2) \, \mu(dz) \bigg]  \,(\mu_{2N}-\mu) (d y_2)} \nonumber \\
  & = & \frac{1}{(2N)^2} \sum_{i,j=1}^{2N}  \frac{\delta^{2} U}{\delta m^{2}}(m^{2N}_t)(\xi_i,\xi_j) - \frac{1}{2N} \sum_{j=1}^{2N} \int_{\bR^d} \frac{\delta^{2} U}{\delta m^{2}}(m^{2N}_t)(z,\xi_j) \, \mu(dz)  \nonumber \\
  && - \frac{1}{2N} \sum_{i=1}^{2N} \int_{\bR^d} \frac{\delta^{2} U}{\delta m^{2}}(m^{2N}_t)(\xi_i,z) \, \mu(dz) + \int_{\bR^d} \int_{\bR^d} \frac{\delta^{2} U}{\delta m^{2}}(m^{2N}_t)(z,z') \, \mu(dz)  \, \mu(dz') \nonumber \\
  & = & \frac{1}{(2N)^2} \sum_{i,j=1}^{2N} \varphi^{(i,j)}_t, \label{eq: second order antithetic difference estimation one term no time discretisation}    \end{eqnarray}
  where
\begin{eqnarray}
   \varphi^{(i,j)}_t  & := &   \frac{\delta^{2} U}{\delta m^{2}}(m^{2N}_t)(\xi_i,\xi_j) -  \int_{\bR^d} \frac{\delta^{2} U}{\delta m^{2}}(m^{2N}_t)(z,\xi_j) \, \mu(dz)  \nonumber \\
  && -  \int_{\bR^d} \frac{\delta^{2} U}{\delta m^{2}}(m^{2N}_t)(\xi_i,z) \, \mu(dz) +  \int_{\bR^d} \int_{\bR^d} \frac{\delta^{2} U}{\delta m^{2}}(m^{2N}_t)(z,z') \, \mu(dz)  \, \mu(dz')  . \nonumber    \end{eqnarray}
 Next, we observe that
\begin{eqnarray}
 && \bE \bigg|\frac{1}{(2N)^2} \sum_{i,j=1}^{2N}  \varphi^{(i,j)}_t \bigg|^2 \nonumber \\
& \lesssim & \frac{1}{N^2} + \frac{1}{N^4} \Bigg[ \sum_{ \substack{i_1, j_1, i_2,j_2 \in \{1, \ldots, 2N \} \\ \text{exactly two of } i_1, j_1, i_2,j_2 \text{ are identical}} } \bE \Big[\varphi^{(i_1,j_1)}_t \varphi^{(i_2,j_2)}_t \Big] +  \sum_{ \substack{i_1, j_1, i_2,j_2 \in \{1, \ldots, 2N \} \\  i_1, j_1, i_2,j_2 \text{ are distinct}} } \bE \Big[\varphi^{(i_1,j_1)}_t \varphi^{(i_2,j_2)}_t \Big] \Bigg]. \nonumber \\
&& \label{eq: interpolation in measures phi i j expansion} 
\end{eqnarray}
  We first consider the case  where exactly two of $i_1, i_2, j_1,j_2$ are identical. Without loss of generality, suppose that $i_1=i_2 $. 
   As in the proof of Lemma \ref{delarue theorem modified}, we define
\begin{eqnarray}
  &&  \varphi^{(i,j),-(i_1,j_1, j_2)}_t  \nonumber \\
  & := &  \frac{\delta^{2} U}{\delta m^{2}}(m^{2N,-(i_1,j_1, j_2)}_t)(\xi_i,\xi_j)   -  \int_{\bR^d} \frac{\delta^{2} U}{\delta m^{2}}(m^{2N,-(i_1,j_1, j_2)}_t)(z,\xi_j) \, \mu(dz)  \nonumber \\
  && -   \int_{\bR^d} \frac{\delta^{2} U}{\delta m^{2}}(m^{2N,-(i_1,j_1, j_2)}_t)(\xi_i,z) \, \mu(dz)   +  \int_{\bR^d} \int_{\bR^d} \frac{\delta^{2} U}{\delta m^{2}}(m^{2N,-(i_1,j_1, j_2)}_t)(z,z') \, \mu(dz)  \, \mu(dz')  ,  \nonumber \\
  && \label{eq: varphi expressions antithetic}    \end{eqnarray}
  where
  $$ m^{2N,-(i_1,j_1, j_2)}_t:= (1-t) \mu + t \bigg[ \frac{1}{2N-3} \sum_{ \substack{1 \leq \ell  \leq 2N \\ \ell \not\in \{ i_1,j_1, j_2 \}}} \delta_{\xi_{\ell}} \bigg].$$ 
   By the same argument as in the proof of Lemma \ref{delarue theorem modified}, along with the bound on $\frac{\delta^3 U}{\delta m^3}$ in \eqref{eq: first order lions linear functional def} (see \eqref{eq: i3 expansion} for details), we have
  $$ \bE |\varphi^{(i,j)}_t - \varphi^{(i,j),-(i_1,j_1, j_2)}_t|^2 \lesssim \frac{1}{N^2}. $$ 
  Then, we write
\begin{eqnarray} 
\bE \Big[\varphi^{(i_1,j_1)}_t \varphi^{(i_1,j_2)}_t \Big]
& = & \bE \Big[(\varphi^{(i_1,j_1)}_t - \varphi^{(i_1,j_1),-(i_1,j_1, j_2)}_t) (\varphi^{(i_1,j_2)}_t- \varphi^{(i_1,j_2),-(i_1,j_1, j_2)}_t) \Big] \nonumber \\
&& + \bE\Big[(\varphi^{(i_1,j_1)}_t - \varphi^{(i_1,j_1),-(i_1,j_1, j_2)}_t) \varphi^{(i_1,j_2),-(i_1,j_1, j_2)}_t \Big] \nonumber \\ 
&& +\bE \Big[  \varphi^{(i_1,j_1),-(i_1,j_1, j_2)}_t (\varphi^{(i_1,j_2)}_t- \varphi^{(i_1,j_2),-(i_1,j_1, j_2)}_t) \Big] \nonumber \\
&& + \bE \Big[  \varphi^{(i_1,j_1),-(i_1,j_1, j_2)}_t \varphi^{(i_1,j_2),-(i_1,j_1, j_2)}_t \Big]. \nonumber 
\end{eqnarray}
By the Cauchy-Schwarz inequality and the bound on $\frac{\delta^2 U}{\delta m^2}$ in \eqref{eq: first order lions linear functional def}, the first three terms converge to $0$ in the order $O(1/N)$. Let $\cF^{-i}$ be the $\sigma$-algebra generated by $\xi_1, \ldots, \xi_N$ except $\xi_i$. Then 
$$\bE \Big[  \varphi^{(i_1,j_1),-(i_1,j_1, j_2)}_t \varphi^{(i_1,j_2),-(i_1,j_1, j_2)}_t \Big] =   \bE \Big[  \varphi^{(i_1,j_1),-(i_1,j_1, j_2)}_t \bE \Big[ \varphi^{(i_1,j_2),-(i_1,j_1, j_2)}_t \Big| \cF^{-j_2}\Big] \Big] =0. $$ 
Therefore,
\begin{equation} \frac{1}{N^4} \sum_{ \substack{i_1, j_1, i_2,j_2 \in \{1, \ldots, 2N \} \\ \text{exactly two of } i_1, j_1, i_2,j_2 \text{ are identical}} } \bE \Big[\varphi^{(i_1,j_1)}_t \varphi^{(i_2,j_2)}_t \Big] \lesssim \frac{1}{N^2}. \label{eq: antithetic interpolation first part }
\end{equation}
Finally, we consider the case where $i_1, j_1,i_2, j_2$ are mutually distinct.  We define $\varphi^{(i,j),-(i_1,j_1, i_2,j_2)}_t$ analogously, as the definition of $\varphi^{(i,j),-(i_1,j_1, j_2)}_t$ in \eqref{eq: varphi expressions antithetic}. As above, we write 
\begin{eqnarray} 
\bE \Big[\varphi^{(i_1,j_1)}_t \varphi^{(i_2,j_2)}_t \Big]
& = & \bE \Big[(\varphi^{(i_1,j_1)}_t - \varphi^{(i_1,j_1),-(i_1,j_1, i_2,j_2)}_t) (\varphi^{(i_2,j_2)}_t- \varphi^{(i_2,j_2),-(i_1,j_1, i_2,j_2)}_t) \Big] \nonumber \\
&& + \bE\Big[(\varphi^{(i_1,j_1)}_t - \varphi^{(i_1,j_1),-(i_1,j_1, i_2,j_2)}_t) \varphi^{(i_2,j_2),-(i_1,j_1, i_2,j_2)}_t \Big] \nonumber \\ 
&& +\bE \Big[  \varphi^{(i_1,j_1),-(i_1,j_1, i_2,j_2)}_t (\varphi^{(i_2,j_2)}_t- \varphi^{(i_2,j_2),-(i_1,j_1, i_2,j_2)}_t) \Big] \nonumber \\
&& + \bE \Big[  \varphi^{(i_1,j_1),-(i_1,j_1, i_2,j_2)}_t \varphi^{(i_2,j_2),-(i_1,j_1, i_2,j_2)}_t \Big]. \nonumber 
\end{eqnarray}
As before, we have
$$ \bE |\varphi^{(i,j)}_t - \varphi^{(i,j),-(i_1,j_1, i_2,j_2)}_t|^2 \lesssim \frac{1}{N^2} $$
and hence
\begin{equation} \bE \Big|(\varphi^{(i_1,j_1)}_t - \varphi^{(i_1,j_1),-(i_1,j_1, i_2,j_2)}_t) (\varphi^{(i_2,j_2)}_t- \varphi^{(i_2,j_2),-(i_1,j_1, i_2,j_2)}_t) \Big| \lesssim \frac{1}{N^2}, \label{eq: last case part i i1 j1 i2 j2} \end{equation}
by the Cauchy-Schwarz inequality. By the same argument as in the proof of Lemma \ref{delarue theorem modified} through considering the fourth order linear functional derivative of $U$, along with the bound on $\frac{\delta^4 U}{\delta m^4}$ in \eqref{eq: first order lions linear functional def} (see \eqref{ fourth order varepsilon third order linear functional derivative} and \eqref{eps fourth moment bound} for details), we obtain that
\begin{eqnarray}
&& \varphi^{(i_1,j_1)}_{t} - \varphi^{(i_1,j_1),-(i_1, j_1,i_2,j_2)}_{t} \nonumber \\
&=&    F_1(\big( \xi_r \big)_{r \neq i_1, j_1,i_2,j_2}, \xi_{i_1}, \xi_{j_1}, \xi_{i_2}) + F_2(\big( \xi_r \big)_{r \neq i_1, j_1,i_2,j_2}, \xi_{i_1}, \xi_{j_1}, \xi_{j_2}) + \tilde{\varepsilon}^{(i_1,j_1),-(i_1, j_1, i_2,j_2)}_N, \nonumber 
\end{eqnarray}for some measurable functions $F_1, F_2:(\bR^d)^{2N-1} \to \bR$, where
$$ \bE \Big|\tilde{\varepsilon}^{(i_1,j_1),-(i_1, j_1, i_2,j_2)}_N \Big|^2 \lesssim \frac{1}{N^4}. $$
By a similar conditioning argument as the proof of Lemma \ref{delarue theorem modified}, 
\begin{eqnarray}
   &&  \bE \Big[ \Big( \varphi^{(i_1,j_1)}_{t} - \varphi^{(i_1,j_1),-(i_1, j_1,i_2,j_2)}_{t} -\tilde{\varepsilon}^{(i_1,j_1),-(i_1, j_1, i_2,j_2)}_N \Big) \varphi^{(i_2,j_2),-(i_1,j_1, i_2,j_2)}_t  \Big] \nonumber \\
    & = & \bE \Big[ F_1(\big( \xi_r \big)_{r \neq i_1, j_1,i_2,j_2}, \xi_{i_1}, \xi_{j_1}, \xi_{i_2})  \bE \Big[ \varphi^{(i_2,j_2),-(i_1,j_1, i_2,j_2)}_t \Big| \cF^{-j_2}  \Big] \Big] \nonumber \\
    && + \bE \Big[ F_2(\big( \xi_r \big)_{r \neq i_1, j_1,i_2,j_2}, \xi_{i_1}, \xi_{j_1}, \xi_{j_2})  \bE \Big[ \varphi^{(i_2,j_2),-(i_1,j_1, i_2,j_2)}_t  \Big| \cF^{-i_2}  \Big] \Big] =0, \nonumber 
\end{eqnarray}
which implies, by the Cauchy-Schwarz inequality and the bound on $\frac{\delta^2 U}{\delta m^2}$ in \eqref{eq: first order lions linear functional def}, that
\begin{equation}
\bE \Big|\Big( \varphi^{(i_1,j_1)}_{t} - \varphi^{(i_1,j_1),-(i_1, j_1,i_2,j_2)}_{t} \Big) \varphi^{(i_2,j_2),-(i_1,j_1, i_2,j_2)}_t  \Big| \lesssim \frac{1}{N^2} . \label{eq: last case part ii i1 j1 i2 j2}
\end{equation}
Similarly, 
\begin{equation}
\bE \Big|  \varphi^{(i_1,j_1),-(i_1,j_1, i_2,j_2)}_t (\varphi^{(i_2,j_2)}_t- \varphi^{(i_2,j_2),-(i_1,j_1, i_2,j_2)}_t) \Big| \lesssim \frac{1}{N^2} . \label{eq: last case part iii i1 j1 i2 j2}
\end{equation}
By the same conditioning argument, 
\begin{equation}
\bE \Big[  \varphi^{(i_1,j_1),-(i_1,j_1, i_2,j_2)}_t \varphi^{(i_2,j_2),-(i_1,j_1, i_2,j_2)}_t \Big] = \bE \Big[  \varphi^{(i_1,j_1),-(i_1,j_1, i_2,j_2)}_t \bE \Big[  \varphi^{(i_2,j_2),-(i_1,j_1, i_2,j_2)}_t \Big| \cF^{-i_2} \Big] \Big] =0. 
\label{eq: last case part iv i1 j1 i2 j2}
\end{equation}
A combination of \eqref{eq: last case part i i1 j1 i2 j2}, \eqref{eq: last case part ii i1 j1 i2 j2}, \eqref{eq: last case part iii i1 j1 i2 j2} and \eqref{eq: last case part iv i1 j1 i2 j2} implies that
\begin{equation}
    \frac{1}{N^4} \sum_{ \substack{i_1, j_1, i_2,j_2 \in \{1, \ldots, 2N \} \\  i_1, j_1, i_2,j_2 \text{ are distinct}} } \bE \Big[\varphi^{(i_1,j_1)}_t \varphi^{(i_2,j_2)}_t \Big] \lesssim \frac{1}{N^2}.   \label{eq: antithetic interpolation second part }
\end{equation}
Finally, a combination of \eqref{eq: second order antithetic difference estimation one term no time discretisation},  \eqref{eq: interpolation in measures phi i j expansion},  \eqref{eq: antithetic interpolation first part } and \eqref{eq: antithetic interpolation second part } implies that
$$ \bE \bigg| {\int_{\bR^{2d}}\frac{\delta^{2} U}{\delta m^{2}}(m^{2N}_t)(\mathbf{y})\,(\mu_{2N}-\mu)^{\otimes 2} (d \mathbf{y})} \bigg|^2 \lesssim \frac{1}{N^2}. $$
    \end{proof}

\section{Dimension-independent rate of uniform strong  propagation of chaos}

    We now introduce a mean-field coupling of the particle system \eqref{eq:particlesystem} by
\begin{equation} \label{eq:coupling eqn} 
\begin{cases} 
      dX^i_t = \xi_i + \int_0^t b(X^i_s, \mclaw[s]) \,ds + \int_0^t \sigma (X^i_s, \mclaw[s]) \,dW^i_s, \quad 1 \leq i \leq N,  \quad  t \in [0, T],\\
      \\
       \couplaw[s] := \frac{1}{N} \sum_{i=1}^N \delta_{X^{i}_s}. 
   \end{cases}
\end{equation}
The following two assumptions are adopted in most results. We assume that 
\[ \begin{cases} 
      b \text{ and } \sigma  \text{ are Lipschitz continuous with respect to the Euclidean norm  and the } W_2 \text{ norm}, & \\
      &  \\
      \Phi \text{ is Lipschitz continuous with respect to the }  W_2 \text{ norm,} &
   \end{cases} \label{eq:Lip} \tag{Lip}
\]
and that the initial law $\nu$ satisfies
\[ \int_{\bR^d} |x|^{12} \, \nu(dx) < +\infty. \quad \label{eq:Int} \tag{Int}
\]
Note that \eqref{eq:Lip} guarantees strong existence and uniqueness of \eqref{eq:MVSDE} and  \eqref{eq:particlesystem}. The following proposition is essential to the proofs of Theorem \ref{strong error} and Theorem \ref{variance antithetic non-discretised}. \begin{proposition} \label{important proposition lemma strong error} Suppose that $b$ and $\sigma$ admit linear growth in the spatial and measure components. Suppose that $\varphi \in \cM_3( \bR^d \times \cP_2(\bR^d))$. Assume \eqref{eq:Int}. Then
$$ \frac{1}{N} \sum_{i=1}^{N} \sup_{t \in [0,T]} \bE \Big| \varphi(X^i_t, \couplaw[t]) - \varphi(X^i_t, \mclaw[t])  \Big|^4 \leq \frac{C}{N^2},$$ 
for some constant $C>0$. 
\end{proposition}
\begin{proof}
\begin{eqnarray}
&& \frac{1}{N} \sum_{i=1}^{N} \sup_{t \in [0,T]} \bE \bigg[ \bigg|\varphi \Big( X^{i}_t,\frac{1}{N} \sum_{j=1}^{N} \delta_{X^j_t} \Big) -  \varphi(X^{i}_t, \mclaw[t])  \bigg|^4 \bigg] \nonumber \\
& = &  \frac{1}{N} \sum_{i=1}^{N} \sup_{t \in [0,T]} \bE \bigg[ \bE \bigg[ \bigg|\varphi \Big( \eta,\frac{1}{N} \delta_{\eta} + \frac{N-1}{N} \cdot \frac{1}{N-1} \sum_{\substack{ 1 \leq j \leq N \\ j \neq i}} \delta_{X^j_t} \Big) -  \varphi(\eta, \mclaw[t])  \bigg|^4 \bigg] \bigg|_{\eta= X^{i}_t} \bigg]  \nonumber \\
& \leq & \frac{8}{N} \sum_{i=1}^{N} \sup_{t \in [0,T]} \bE \bigg[ \bE \bigg[ \bigg|\varphi \Big( \eta,\frac{1}{N} \delta_{\eta} + \frac{N-1}{N} \cdot \frac{1}{N-1} \sum_{\substack{ 1 \leq j \leq N \\ j \neq i}} \delta_{X^j_t} \Big) \nonumber \\
&& -  \varphi \Big(\eta, \frac{1}{N-1} \sum_{\substack{ 1 \leq j \leq N \\ j \neq i}} \delta_{X^j_t} \Big)  \bigg|^4 \bigg] \bigg|_{\eta= X^{i}_t} \bigg]  \nonumber \\
&& + \frac{8}{N} \sum_{i=1}^{N} \sup_{t \in [0,T]} \bE \bigg[ \bE \bigg[ \bigg|\varphi \Big( \eta, \frac{1}{N-1} \sum_{\substack{ 1 \leq j \leq N \\ j \neq i}} \delta_{X^j_t} \Big) -  \varphi(\eta, \mclaw[t])  \bigg|^4 \bigg] \bigg|_{\eta= X^{i}_t} \bigg]  \nonumber \\
& =: & \Pi_1 + \Pi_2. \nonumber 
\end{eqnarray} 
By Lemma \ref{W2 lemma}, using the same type of estimate as \eqref{X Yn bound}, we have
\begin{equation}
\Pi_1  \leq   \frac{8}{N} \sum_{i=1}^{N} \sup_{t \in [0,T]} \bE \bigg[ \frac{4}{N^2} \bigg(|X^i_t|^2 + \frac{1}{N-1} \sum_{\substack{ 1 \leq j \leq N \\ j \neq i}}   |X^j_t|^2 \bigg)^2 \bigg] \lesssim \frac{1}{N^2} . \label{Pi 1 bound} \end{equation}
By the assumption on $\varphi, $ we observe that for any $\eta \in \bR^d$, the uniform bounds on $\pmu \varphi(\eta, \cdot)$, $\ptwomu \varphi (\eta, \cdot)$  and $\pthreemu \varphi (\eta, \cdot)$ do not depend on $\eta$. Finally, since $b$ and $\sigma$ are of linear growth in the spatial and measure components and $\bE[|\xi|^{12}]< +\infty$, we have $\sup_{t \in [0,T]} \bE[|X_t|^{12}]< +\infty$. Therefore, Lemma \ref{delarue theorem modified} implies that 
\begin{equation} \Pi_2 \lesssim \frac{1}{(N-1)^2} \prod_{i=1}^3 \Big( 1+ \sup_{\eta \in \bR^d} \| \partial^i_{\mu} \varphi(\eta, \cdot) \|^4_{\infty}  \Big) \bigg( 1+ \sup_{t \in [0,T]} \int_{\bR^d} |y|^{12} \, \mclaw[t] (dy) \bigg).  \label{Pi 2 bound} \end{equation} 
A combination of \eqref{Pi 1 bound} and \eqref{Pi 2 bound} yields the result. 
\end{proof}
Note that Proposition \ref{important proposition lemma strong error} allows us to completely bypass the consideration of the Wasserstein distance between empirical measures and their limiting law. Assuming \eqref{eq:Lip} and \eqref{eq:Int}, Theorem 10.2.7 in \cite{rachev1998mass} gives us a rate of convergence of 
    \begin{equation}  \bE \Big[ \sup_{t \in [0,T]} W_2 \big( \mclaw[t], \couplaw[t] \big)^2 \Big] \leq \frac{C}{N^{2/(d+8)}}.  \label{eq: decom W2 first term} \end{equation}

The following result gives a uniform rate of strong propagation of chaos between the particle system \eqref{eq:particlesystem} and its coupled mean-field limit \eqref{eq:coupling eqn}, under the assumption that $b$ and $\sigma$ are sufficiently smooth in the sense of L-derivatives. This is a different set of sufficient conditions compared to the existing results in the literature with the same rate, such as Lemma 5.1 in \cite{delarue2019master}, Theorem 1 in \cite{jabin2018quantitative} and \cite{jabir2019rate}.

Let $\cC_T:= C([0,T],\bR^d)$ be the space of continuous functions from $[0,T]$ to $\bR^d$ equipped with the supremum norm and $W_{\cC_T,2}$ be the 2-Wasserstein metric on $\cC_T $. 
\begin{theorem}[Uniform strong propagation of chaos] \label{strong error}
Assume \eqref{eq:Int}. Suppose that $b, \sigma \in \cM_3( \bR^d \times \cP_2(\bR^d))$. Then
$$ \bE \Big[ W_{\cC_T,2} \big( \nlaw[], \couplaw[] \big)^4 \Big] \leq \bE \bigg[ \frac{1}{N} \sum_{i=1}^N  \bigg( \sup_{t \in [0,T]} \big| X^i_{t} - Y^{i,N}_{t} \big|^4 \bigg) \bigg]  \leq \frac{C}{N^2} , $$ 
for some constant $C>0$.
\end{theorem}
\begin{proof}
By the H\"{o}lder and Buckholder-Davis-Gundy inequalities, estimating the $L^4$ difference between \eqref{eq:particlesystem} and \eqref{eq:coupling eqn}  gives
   \begin{eqnarray}  \bE \Big[ \sup_{s \in [0,t]} \big| X^i_{s} - Y^{i,N}_{s} \big|^4 \Big]  & \leq & 
     C \bigg( \int_0^{t}   \bE |b(X^i_s, \mclaw[s]) - b(Y^{i,N}_s, \nlaw[s]) |^4 \,ds \nonumber \\
     && + \int_0^{t} \bE \|\sigma(X^i_s, \mclaw[s]) - \sigma(Y^{i,N}_s, \nlaw[s]) \|^4 \,ds \bigg), \label{X Yn bound} 
     \end{eqnarray}
    for every $t \in [0,T]$. 
      By Lipschitz continuity of $b$ and $\sigma$,
   \begin{eqnarray}  \bE \Big[ \sup_{s \in [0,t]} \big| X^i_{s} - Y^{i,N}_{s} \big|^4 \Big]  & \leq & 
     C \bigg( \int_0^{t}  \bE \Big[ \sup_{u \in [0,s]} \big| X^i_{u} - Y^{i,N}_{u} \big|^4 \Big]  \,ds  + \int_0^{t}   \bE |b(X^i_s, \mclaw[s]) - b(X^i_s, \nlaw[s]) |^4 \,ds  \nonumber \\
     && + \int_0^{t} \bE \|\sigma(X^i_s, \mclaw[s]) - \sigma(X^i_s, \nlaw[s]) \|^4 \,ds \bigg), \nonumber  
     \end{eqnarray}
    for every $t \in [0,T]$, which gives, upon taking average over $i$,
    \begin{eqnarray}  \frac{1}{N} \sum_{i=1}^N \bE \Big[ \sup_{s \in [0,t]} \big| X^i_{s} - Y^{i,N}_{s} \big|^4 \Big]   & \leq & 
     C \bigg( \int_0^{t}   \frac{1}{N} \sum_{i=1}^N  \bE \Big[ \sup_{u \in [0,s]} \big| X^i_{u} - Y^{i,N}_{u} \big|^4 \Big]  \,ds  \nonumber \\
     && + \int_0^{t}   \frac{1}{N} \sum_{i=1}^N  \bE |b(X^i_s, \mclaw[s]) - b(X^i_s, \nlaw[s]) |^4 \,ds  \nonumber \\
     && + \int_0^{t}  \frac{1}{N} \sum_{i=1}^N  \bE \|\sigma(X^i_s, \mclaw[s]) - \sigma(X^i_s, \nlaw[s]) \|^4 \,ds \bigg). \label{eq: lipschitz est particles couplings initial} 
     \end{eqnarray}
        Also, the empirical measure of the particles can be replaced by the empirical measure of the coupled system by the bound
    \begin{eqnarray}  \bE \big[ W_2 ( \couplaw[s], \nlaw[s])^4 \big]  \leq  \bigg[ \bigg( \frac{1}{N} \sum_{i=1}^N  \bE  \big| Y^{i,N}_s - X^i_s \big|^2 \bigg)^2 \bigg] \leq \frac{1}{N} \sum_{i=1}^N  \bE \Big[ \sup_{u \in [0,s]} \big| X^i_{u} - Y^{i,N}_{u} \big|^4 \Big].  \label{eq: decom W2 second term}
       \end{eqnarray} 
A combination of \eqref{eq: lipschitz est particles couplings initial} and \eqref{eq: decom W2 second term} gives
 \begin{eqnarray}  
 \frac{1}{N} \sum_{i=1}^N \bE \Big[ \sup_{s \in [0,t]} \big| X^i_{s} - Y^{i,N}_{s} \big|^4 \Big]  & \leq & 
     C \bigg( \int_0^{t} \frac{1}{N} \sum_{i=1}^N  \bE \Big[ \sup_{u \in [0,s]} \big| X^i_{u} - Y^{i,N}_{u} \big|^4 \Big]  \,ds   \nonumber \\
     && + \int_0^{t} \frac{1}{N} \sum_{i=1}^N \sup_{u \in [0,s]}  \bE |b(X^i_u, \mclaw[u]) - b(X^i_u, \couplaw[u]) |^4 \,ds  \nonumber \\
     && + \int_0^{t} \frac{1}{N} \sum_{i=1}^N \sup_{u \in [0,s]}  \bE \|\sigma(X^i_u, \mclaw[u]) - \sigma(X^i_u, \couplaw[u]) \|^4 \,ds \bigg). \nonumber  
     \end{eqnarray}
     Therefore, by Proposition \ref{important proposition lemma strong error} and Gronwall's inequality, we have
     $$ \frac{1}{N} \sum_{i=1}^N \bE \Big[ \sup_{s \in [0,T]} \big| X^i_{s} - Y^{i,N}_{s} \big|^4 \Big] \leq  \frac{C}{N^2}, $$
     for every $t \in [0,T]$. 
\end{proof}

\section{Antithetic MLMC without time discretisation}  \label{theoretical MLMC} 

    The main aim of this section is to prove the complexity of the antithetic MLMC estimator, via the following theorem, which states that the variance of the antithetic difference in \eqref{eq AMLMC} converges in $N$ in the rate $O(1/N^2)$. In the proof, Proposition \ref{important proposition lemma strong error} and Theorem \ref{strong error} provide us with the necessary estimates when we revert to the mean-field limit. 
\begin{theorem}[Variance of antithetic difference] \label{variance antithetic non-discretised} 
Assume \eqref{eq:Int}. Suppose that $b, \sigma \in \cM_4 \big( \bR^d  \times \mathcal{P}_2(\bR^d) \big) $ and $\Phi \in \cM_4 \big(  \mathcal{P}_2(\bR^d) \big) $.  Then 
$$\text{Var} \Big[ \Phi (\twonlaw[T]) -  \frac{1}{2} \big( \Phi (\twonlawone[T]) + \Phi (\twonlawtwo[T]) \big) \Big] \leq   \frac{C}{N^2}, $$ 
where $C$ is a constant that depends on $\Phi$, $b$, $\sigma$ and $T$, but does not depend on $N$.
\end{theorem}
%

\begin{proof}[Proof of Theorem \ref{variance antithetic non-discretised}]
The main techniques in the proof depend on  the function $\cV : [0,T] \times \mathcal{P}_2 (\bR^d) \to \bR $, which is defined in \eqref{eq: defofflow} by 
\begin{equation*}  \cV( s, \rvlaw[\eta]) = \Phi \big( \rvlaw[{X^{s, \eta}_T}] \big). 
\end{equation*}
Another crucial ingredient in the proof is \eqref{eq representation}, which represents the difference  $ \Phi (\nlaw[T])  - \Phi (\mclaw[T]) $ as 
\begin{eqnarray} 
     \Phi (\nlaw[T])  - \Phi (\mclaw[T])  & =  &  \big( \cV(0,\nlaw[0]) - \cV(0,\nu) \big)  \nonumber \\
     && +  \int_0^{T} \frac{1}{2} \Bigg[\frac{1}{N^2} \sum_{i=1}^N  \text{Tr} \bigg( a \big(Y^{i,N}_s, \nlaw[s]  \big)    \partial^2_{\mu}  {\cV} \big( s, \nlaw[s]  \big)  (Y^{i,N}_s, Y^{i,N}_s) \bigg) \Bigg] \,ds \nonumber  \\
&& +\frac{1}{N}\sum_{i=1}^N   \int_0^{T}  \sigma( Y^{i,N}_s, \nlaw[s] )^T \partial_{\mu} {\cV} \big( s, \nlaw[s]  \big) ( Y^{i,N}_s) \cdot dW_s^i. \nonumber 
\end{eqnarray}
Hence,  
\begin{equation*} 
\begin{split}
    & \Phi (\twonlaw[T]) -  \frac{1}{2} \big( \Phi (\twonlawone[T]) + \Phi (\twonlawtwo[T]) \big) 
     =      \mathscr{A} + \mathscr{D} + \mathscr{S},
\end{split} 
\end{equation*}
where 
$$ \mathscr{A} := \cV(0,\twonlaw[0]) - \frac{1}{2} \big(\cV(0,\twonlawone[0]) + \cV(0,\twonlawtwo[0]) \big),  $$ 
\begin{eqnarray*}
\mathscr{D} & := & \int_0^{T} \frac{1}{2}  \Bigg[\frac{1}{(2N)^2} \sum_{i=1}^{2N}  \text{Tr} \bigg( a \big(Y^{i,2N}_s, \twonlaw[s]  \big)    \partial^2_{\mu}  {\cV} \big( s, \twonlaw[s]  \big)  (Y^{i,2N}_s, Y^{i,2N}_s) \bigg) \Bigg]  \\
&& -\frac{1}{2N^2}  \bigg[   \sum_{i=1}^{N}  \text{Tr} \bigg( a \big(Y^{i,2N,(1)}_s, \twonlawone[s]  \big)    \partial^2_{\mu}  {\cV} \big( s, \twonlawone[s]   \big)  (Y^{i,2N,(1)}_s, Y^{i,2N,(1)}_s) \bigg)  \\
&& + \sum_{i=N+1}^{2N}  \text{Tr} \bigg( a \big(Y^{i,2N,(2)}_s, \twonlawtwo[s]   \big)    \partial^2_{\mu}  {\cV} \big( s, \twonlawtwo[s]  \big)  (Y^{i,2N,(2)}_s, Y^{i,2N,(2)}_s) \bigg) \bigg] \, ds
\end{eqnarray*}
and
\begin{eqnarray*}
 \mathscr{S} & := & \sum_{i=1}^{2N}  \int_0^{T}  \frac{1}{2N} \partial_{\mu} {\cV} \big( s, \twonlaw[s] \big) (  Y^{i,2N}_s)^T  \sigma( Y^{i,2N}_s,  \twonlaw[s] )dW_s^i \\
&& -\frac{1}{2N} \bigg(\sum_{i=1}^{N} \int_0^T \partial_{\mu} {\cV} \big( \twonlawone[s] \big) (  Y^{i,2N,(1)}_s)^T  \sigma( Y^{i,2N,(1)}_s,  \twonlawone[s] )dW_s^i \\ 
&& + \sum_{i=N+1}^{2N}  \int_0^T \partial_{\mu} {\cV} \big( s, \twonlawtwo[s] \big) (  Y^{i,2N,(2)}_s)^T  \sigma( Y^{i,2N,(2)}_s, \twonlawtwo[s] )dW_s^i \bigg).
\end{eqnarray*}
By the assumptions on $b$,  $\sigma$ and $\Phi$, it follows from Theorem  \ref{eq:generalisationmainresult} that $ \cV \in \cM_4 \big(  [0,T] \times \mathcal{P}_2(\bR^d) \big)$. We can therefore see that 
\[
\bE[ \mathscr{D}^2 ] \lesssim 1/N^2.
\]
In particular, $ \cV(0,\cdot) \in \cM_4 \big(  \mathcal{P}_2(\bR^d) \big)$. Therefore, by  Theorem \ref{antithetic initial separation}, we obtain that
 \[
\bE[ \mathscr{A}^2 ] \lesssim 1/N^2.
\]
Hence, it remains to show that $\bE(\mathscr{S}^2)\lesssim 1/N^2$. Define $\Sigma(t,x,\mu):= \partial_{\mu} {\cV} \big( t, \mu \big) ( x )^T \sigma( x, \mu ).$ By the independence of the Brownian motions, we first rewrite $\bE[\mathscr{S}^2]$ as
\begin{eqnarray}
 \bE[\mathscr{S}^2]
 & = &\bE \bigg[ \bigg( \frac{1}{2N} \sum_{i=1}^{N} \int_0^{T} 
 \Sigma(s,Y^{i,2N}_s,\twonlaw[s])- \Sigma(s,Y^{i,2N,(1)}_s,\twonlawone[s])dW_s^i \bigg)^2 \bigg] \nonumber \\
&& +
\bE \bigg[ \bigg( \frac{1}{2N}  \sum_{i=N+1}^{2N} \int_0^{T}   \Sigma(s,Y^{i,2N}_s,\twonlaw[s])- \Sigma(s,Y^{i,2N,(2)}_s,\twonlawtwo[s])dW_s^i \bigg)^2 \bigg]. \nonumber 
\end{eqnarray}
Using the independence of the Brownian motions and It\^{o}'s isometry, 
\begin{equation*}
\begin{split}
 & \bE \bigg[ \bigg(\frac{1}{2N} \sum_{i=1}^{N}  \int_0^{T} 
 \Sigma(s,Y^{i,2N}_s,\twonlaw[s])- \Sigma(s,Y^{i,2N,(1)}_s,\twonlawone[s])dW_s^i \bigg)^2 \bigg] \\
& = \frac{1}{4N^2}  \sum_{i=1}^{N} \bE \bigg[ \bigg( \int_0^{T} 
 \Sigma(s,Y^{i,2N}_s,\twonlaw[s])- \Sigma(s,Y^{i,2N,(1)}_s,\twonlawone[s])dW_s^i \bigg)^2 \bigg]  \\
& = \frac{1}{4N^2}  \sum_{i=1}^{N}  \int_0^{T} 
\bE  \Big[ \Big\| \Sigma(s,Y^{i,2N}_s,\twonlaw[s])- \Sigma(s,Y^{i,2N,(1)}_s,\twonlawone[s]) \Big\|^2 \Big] \, ds.  
\end{split}
\end{equation*}
Note that $\cV \in \cM_4([0,T] \times \cP_2(\bR^d))$. Therefore, $\pmu \cV$ is Lipschitz continuous and uniformly bounded. Also, note that  $\sigma$ is Lipschitz continuous. By Theorem \ref{strong error}, 
\begin{eqnarray} 
&& \sup_{t \in [0,T]} \bE \big[ \big\| \Sigma(t,Y^{i,2N}_t, \twonlaw[t])- \Sigma ( t,X^{i}_t,\twocouplaw[t] ) \big \|^2 \big] \nonumber \\
& = &  \sup_{t \in [0,T]} \bE \big[ \big\| \pmu \cV(t, \twonlaw[t])(Y^{i,2N}_t)^T \sigma(Y^{i,2N}_t,\twonlaw[t]) -  \pmu \cV(t, \twocouplaw[t])(X^{i}_t)^T \sigma( X^{i}_t, \twocouplaw[t]) \big \|^2 \big]  \nonumber \\
& \lesssim & \sup_{t \in [0,T]} \bE \big[ \big\| \pmu \cV(t, \twonlaw[t])(Y^{i,2N}_t)^T \big( \sigma(Y^{i,2N}_t,\twonlaw[t]) -   \sigma( X^{i}_t, \twocouplaw[t]) \big) \big \|^2 \big]  \nonumber \\
&& +  \sup_{t \in [0,T]} \bE \big[ \big\| \big( \pmu \cV(t, \twonlaw[t])(Y^{i,2N}_t)^T -  \pmu \cV(t, \twocouplaw[t])(X^{i}_t)^T \big) \sigma( X^{i}_t, \twocouplaw[t]) \big \|^2 \big]  \nonumber \\
& \lesssim & \sup_{t \in [0,T]} \bE \big[ \big\| \pmu \cV(t, \twonlaw[t])(Y^{i,2N}_t)^T \big( \sigma(Y^{i,2N}_t,\twonlaw[t]) -   \sigma( X^{i}_t, \twocouplaw[t]) \big) \big \|^2 \big]  \nonumber \\
&& +  \sup_{t \in [0,T]} \Big( \bE \big[ \big\| \big( \pmu \cV(t, \twonlaw[t])(Y^{i,2N}_t) -  \pmu \cV(t, \twocouplaw[t])(X^{i}_t) \big\|^4 \big] \Big)^{1/2} \Big( \bE \big[ \big\| \sigma( X^{i}_t, \twocouplaw[t]) \big \|^4 \big] \Big)^{1/2}  \nonumber \\
& \lesssim & \sup_{t \in [0,T]} \bE \big[ \big|  Y^{i,2N}_t-  X^{i}_t |^2 \big] + \frac{1}{2N} \sum_{j=1}^{2N} \sup_{t \in [0,T]} \bE \big[ \big|  Y^{j,2N}_t-  X^{j}_t |^2 \big]  \nonumber \\
&& + \Big( \sup_{t \in [0,T]} \bE \big[ \big|  Y^{i,2N}_t-  X^{i}_t |^4 \big] \Big)^{1/2} + \Big( \frac{1}{2N} \sum_{j=1}^{2N} \sup_{t \in [0,T]} \bE \big[ \big|  Y^{j,2N}_t-  X^{j}_t |^4\big]  \Big)^{1/2} \lesssim \frac{1}{N}. 
\label{eq: main bound 1}
\end{eqnarray} 
Similarly, we can show that
\begin{equation} \label{eq: main bound 2}
 \sup_{t \in [0,T]} \bE \big[ \big\| \Sigma(t,Y^{i,2N,(1)}_t, \twonlawone[t])- \Sigma ( t,X^{i}_t,\couplaw[t] ) \big \|^2 \big] \lesssim \frac{1}{N}.
\end{equation} 
Next, we apply Proposition \ref{important proposition lemma strong error} to $\sigma$ and $ \pmu \cV(t. \cdot)(\cdot)$. (Note that the constant $C$ in Proposition \ref{important proposition lemma strong error} corresponding to $\varphi = \pmu \cV(t, \cdot)(\cdot)$ does not depend on time, since the first, second and third order derivatives in measure of this function are uniformly bounded in time.) By a similar calculation as \eqref{eq: main bound 1}, we obtain that

\begin{eqnarray} 
&& \sup_{t \in [0,T]} \bE \big[ \big\|\Sigma ( t,X^{i}_t, \twocouplaw[t] ) -  \Sigma(t,X^{i}_t, \mclaw[t])  \big\|^2 \big]\nonumber \\
& \lesssim &  \sup_{t \in [0,T]} \bE \big[ \big\| \pmu \cV(t, \twocouplaw[t])(X^{i}_t)^T \big( \sigma(X^{i}_t,\twocouplaw[t]) -   \sigma( X^{i}_t, \mclaw[t]) \big) \big \|^2 \big]  \nonumber \\
&& +  \sup_{t \in [0,T]} \Big( \bE \big[ \big\| \big( \pmu \cV(t, \twocouplaw[t])(X^{i}_t) -  \pmu \cV(t, \mclaw[t])(X^{i}_t) \big\|^4 \big] \Big)^{1/2} \Big( \bE \big[ \big\| \sigma( X^{i}_t, \mclaw[t]) \big \|^4 \big] \Big)^{1/2} \lesssim \frac{1}{N}. \nonumber \\ && \label{eq: main bound 3}
\end{eqnarray}  
Similarly,
\begin{equation} \label{eq: main bound 4}
\sup_{t \in [0,T]} \bE \big[ \big\|\Sigma ( t,X^{i}_t, \couplaw[t] ) -  \Sigma(t,X^{i}_t, \mclaw[t])  \big\|^2 \big] \lesssim \frac{1}{N}.
\end{equation}
A combination of \eqref{eq: main bound 1}, \eqref{eq: main bound 2}, \eqref{eq: main bound 3} and \eqref{eq: main bound 4} gives 
$$ \bE \bigg[ \bigg( \frac{1}{2N} \sum_{i=1}^{N} \int_0^{T} 
 \Sigma(s,Y^{i,2N}_s,\twonlaw[s])- \Sigma(s,Y^{i,2N,(1)}_s,\twonlawone[s])dW_s^i \bigg)^2 \bigg] \lesssim \frac{1}{N^2}.$$ 
 Similarly,
 $$ \bE \bigg[ \bigg( \frac{1}{2N}  \sum_{i=N+1}^{2N} \int_0^{T}   \Sigma(s,Y^{i,2N}_s,\twonlaw[s])- \Sigma(s,Y^{i,2N,(2)}_s,\twonlawtwo[s])dW_s^i \bigg)^2 \bigg] \lesssim \frac{1}{N^2}.$$ 
 Consequently, $\bE[\mathscr{S}^2] \lesssim \frac{1}{N^2}.$
\end{proof}
We now perform an analysis on the order of interactions of this algorithm by assuming that $b$ and $\sigma$ are of the forms \eqref{pth order b} and \eqref{pth order sigma} respectively. Recall that, by Theorem \ref{weak error Mckean particles},
 \[
| \bE[\Phi(\nlawell{T}{\ell})] - \Phi(\mclaw[T]) | \leq \frac{C}{N_{\ell}}. \label{eq:i} \tag{i}
\]
Moreover, by Theorem \ref{variance antithetic non-discretised}, we have
\[
\text{Var} \Big[ \Phi(\nlawelli{T}{\ell}{\theta}{\ell}) - \frac{1}{2} \Big( \Phi(\nlawoneelli{T}{\ell}{\theta}{\ell}) + \Phi(\nlawtwoelli{T}{\ell}{\theta}{\ell}) \Big) \Big] \leq \frac{C}{N^2_{\ell}}. \label{eq:ii} \tag{ii}
\]
By Definition \ref{def interactions complexity}, the order of interactions of the antithetic difference is bounded by
\[
\text{Order of interactions}  \Big[ \Phi(\nlawelli{T}{\ell}{\theta}{\ell}) - \frac{1}{2} \Big( \Phi(\nlawoneelli{T}{\ell}{\theta}{\ell}) + \Phi(\nlawtwoelli{T}{\ell}{\theta}{\ell}) \Big) \Big]  \leq CN_{\ell}^{p+1}. \label{eq:iii} \tag{iii}
\]

Properties \eqref{eq:i} to \eqref{eq:iii} allow us to conclude the order of interactions of the theoretical antithetic MLMC estimator. 
\begin{theorem}[Order of interactions of the  theoretical antithetic MLMC estimator \eqref{eq AMLMC}] \label{complexity no time discret} 
Assume \eqref{eq:Int}. Suppose that $b$ and $\sigma$ are of the forms \eqref{pth order b} and \eqref{pth order sigma} respectively. Furthermore, suppose that $b, \sigma \in \cM_4 \big( \bR^d  \times \mathcal{P}_2(\bR^d) \big) $ and $\Phi \in \cM_4 \big(  \mathcal{P}_2(\bR^d) \big) $.  Then there exist constants $C_1, C_2>0$ such that for any $\eps < e^{-1}$, there exist a value $L$ and a sequence $\{ M_{\ell} \}_{\ell=0}^L$ such that the mean-square error of $ {\cA}^{\text{A- MLMC}}$ $($given by \eqref{eq AMLMC}$)$ is bounded by
$$  \bE \big[ \big(  {\cA}^{\text{A- MLMC}} - \Phi(\mclaw[T]) \big)^2 \big]  \leq C_1 \eps^2 $$ 
and the order of interactions of $ {\cA}^{\text{A- MLMC}}$  is bounded by 
\[ \text{\emph{Order of interactions}} \, \big(  {\cA}^{\text{A- MLMC}} \big) \leq        \begin{cases} 
      C_2 \eps^{-2} (\log \eps)^2, & p=1, \\
      C_2 \eps^{-1-p}, & p>1.
   \end{cases}
\]
\end{theorem}
\begin{proof}
 The proof of this theorem is almost identical to the proof of Theorem 1 in \cite{cliffe2011multilevel} and is therefore omitted. Nonetheless, the proof for the complexity of the antithetic MLMC estimator with time discretisation (Theorem \ref{MLMC time discret complexity}) will be presented in detail for completeness. 
\end{proof}
\section{Antithetic MLMC with Euler time discretisation} \label{practical MLMC}
In this section, we construct an MLMC estimator in the same way as the previous section, but with time discretisation. We set 
$$ N_{\ell}:=2^{\ell}, \quad \quad h_{\ell}:= \frac{T}{N_{\ell}}, \quad \quad \ell \in \{0, \ldots, L \}.$$ 
We also set the two sub-particle systems to have the same number of particles. We define the pair of sub-particle systems to $\{Z^{i,2N,h}\}_{i=1}^{2N}$ as
\begin{equation*} 
\begin{split}
       & Z^{i,2N,(1),h}_t  =  \xi_i + \int_0^t  b \bigg(Z^{i,2N,(1),h}_{\eta(r)} , \etaeulertwonlawone[r] \bigg)  \,dr + \int_0^t  \sigma \bigg( Z^{i,2N,(1),h}_{\eta(r)},\etaeulertwonlawone[r] \bigg)    \,  dW^i_r, \quad 1 \leq i \leq N,  \\
    & Z^{i,2N,(2),h}_t  = \xi_i + \int_0^t  b \bigg( Z^{i,2N,(2),h}_{\eta(r)}, \etaeulertwonlawtwo[r] \bigg)  \,dr + \int_0^t  \sigma \bigg( Z^{i,2N,(2),h}_{\eta(r)},\etaeulertwonlawtwo[r]  \bigg)    \,  dW^i_r, \quad N+1 \leq i \leq 2N,
\end{split}
    \end{equation*}
    where
    $$ \eulertwonlawone[r]:= \frac{1}{N} \sum_{i=1}^N \delta_{Z^{i,2N,(1),h}_r} \quad \quad \text{ and } \quad \quad \eulertwonlawtwo[r]:= \frac{1}{N} \sum_{i=N+1}^{2N} \delta_{Z^{i,2N,(2),h}_r}. $$ 
    Therefore, we define the  MLMC estimator with time discretisation as 
    \begin{eqnarray}
        {\cA}^{\text{A- MLMC,$t$}}
       & := & \frac{1}{M_0} \sum_{\theta=1}^{M_0} \Phi(\eulernlawelli{T}{0}{\theta}{0}{h_0}) \nonumber \\
       && + \sum_{\ell=1}^L \bigg[ \frac{1}{M_{\ell}} \sum_{\theta=1}^{M_{\ell}} \Big[ \Phi(\eulernlawelli{T}{\ell}{\theta}{\ell}{h_{\ell}}) - \frac{1}{2} \Big( \Phi(\eulernlawoneelli{T}{\ell}{\theta}{\ell}{2h_{\ell}}) + \Phi(\eulernlawtwoelli{T}{\ell}{\theta}{\ell}{2h_{\ell}}) \Big) \Big] \bigg], \quad \quad \quad \label{eq: MLMC def }
    \end{eqnarray}
    where $\eulernlawelli{T}{\ell}{\theta}{\ell}{h_{\ell}}$, $\eulernlawoneelli{T}{\ell}{\theta}{\ell}{2h_{\ell}}$ and $ \eulernlawtwoelli{T}{\ell}{\theta}{\ell}{2h_{\ell}}$ are defined similarly as $\eulernlawellh{T}{\ell}{h_{\ell}}$, $\eulernlawoneellh{T}{\ell}{2h_{\ell}}$, and $\eulernlawtwoellh{T}{\ell}{2h_{\ell}}$ respectively, but correspond to the $\sum_{\ell=0}^L M_{\ell}$ independent {clouds} of particles indexed by $\ell \in \{0, \ldots, L \}$ and $\theta \in \{ 1, \ldots, M_{\ell} \}$. Each {cloud} (indexed by  $\ell$, $\theta$) has particles with  initial conditions $\xi_{i,\ell, \theta }$, $i \in \{1, \ldots, N_{\ell} \}$,  driven by Brownian motions $W^{i,\ell, \theta},$ $i \in \{1, \ldots, N_{\ell} \}$, where $\{ \xi_{i,\ell, \theta} \} $ and $\{ W^{i,\ell, \theta} \}$ are independent over $i$, $\ell$ and $\theta$. 
    
    To prove the analogue of Theorem \ref{variance antithetic non-discretised} with time discretisation, we need the following lemma that provides a strong error bound between the particle system \eqref{eq:particlesystem} and the Euler scheme \eqref{eq:Euler}. Since we require a higher-order approximation in time discretisation, we restrict ourselves to the case of constant diffusion, in order to avoid the complication of introducing the Milstein scheme of time discretisation. Note that, under \eqref{eq:Lip}, it follows by a standard Gronwall-type argument that
\begin{equation} \label{eq:boundedinL2}
      \sup_{N \in \bN} \sup_{u \in [0,T]} \bE \bigg[ \frac{1}{N} \sum_{i=1}^N |Y^{i,N}_{u}|^2 \bigg] < + \infty,  \quad  \sup_{N \in \bN} \sup_{u \in [0,T]} \bE \bigg[ \frac{1}{N} \sum_{i=1}^N |Z^{i,N,h}_{\eta(u)}|^2 \bigg] < + \infty,
\end{equation}
for some $C>0$.

    \begin{lemma} \label{YZ strong bound W2} 
     Suppose that $b \in \cM_2(\bR^d \times \cP_2(\bR^d))$ and $\sigma$ is constant. Then
    $$ \sup_{N \in \bN} \sup_{s \in [0,T]}\bE \big[ W_2(\nlaw[s], \eulerlaw[s])^2] \leq C h^2,$$
    for some constant $C$ that does not depend on $h$.
    \end{lemma}
    \begin{proof}
     The proof is presented in dimension one, for simplicity of notations. By It\^{o}'s formula,
     $$ (Y^{i,N}_t - Z^{i,N,h}_t)^2 = 2 \int_0^t (Y^{i,N}_s - Z^{i,N,h}_s)\big( b(Y^{i,N}_s, \nlaw[s]) - b(Z^{i,N,h}_{\eta(s)}, \etaeulerlaw[s]) \big) \,ds.$$ 
     Take $0 \leq t' \leq t \leq T$. Then
     \begin{eqnarray}
     \frac{1}{N} \sum_{i=1}^N \bE (Y^{i,N}_{t'} - Z^{i,N,h}_{t'})^2 & = & \frac{2}{N} \sum_{i=1}^N \bE \bigg[ \int_0^{t'} (Y^{i,N}_s - Z^{i,N,h}_s)\big( b(Y^{i,N}_s, \nlaw[s]) - b(Z^{i,N,h}_{s}, \eulerlaw[s]) \big) \,ds \bigg] \nonumber \\
     & & + \frac{2}{N} \sum_{i=1}^N \bE \bigg[ \int_0^{t'} (Y^{i,N}_s - Z^{i,N,h}_s)\big(  b(Z^{i,N,h}_{s}, \eulerlaw[s]) - b(Z^{i,N,h}_{\eta(s)}, \etaeulerlaw[s]) \big) \,ds \bigg]. \nonumber \\
     && \label{eq: ito strong error YZ}
     \end{eqnarray}
     We first bound the first term of \eqref{eq: ito strong error YZ}. 
     \begin{eqnarray}
    &&  \frac{2}{N} \sum_{i=1}^N \bE \bigg[ \int_0^{t'} (Y^{i,N}_s - Z^{i,N,h}_s)\big( b(Y^{i,N}_s, \nlaw[s]) - b(Z^{i,N,h}_{s}, \eulerlaw[s]) \big) \,ds \bigg] \nonumber \\
    & \leq & C \bE \bigg[ \frac{1}{N} \sum_{i=1}^N \int_0^{t'} |Y^{i,N}_s - Z^{i,N,h}_s| \bigg( |Y^{i,N}_s - Z^{i,N,h}_s| + \Big( \frac{1}{N} \sum_{j=1}^N |Y^{i,N}_s - Z^{i,N,h}_s|^2 \Big)^{1/2} \bigg) \,ds \bigg] \nonumber \\  & \leq & \frac{C}{N} \sum_{i=1}^N \int_0^{t'} \bE |Y^{i,N}_s - Z^{i,N,h}_s|^2 \,ds \leq C \int_0^t  \sup_{u \in [0,s]} \bigg[  \frac{1}{N} \sum_{i=1}^N \bE |Y^{i,N}_u - Z^{i,N,h}_u|^2 \bigg] \,ds. \label{eq:first term strong error YZ}  
     \end{eqnarray}
     To bound the second term of \eqref{eq: ito strong error YZ}, we proceed as in the proof of Theorem \ref{weak error Mckean euler} (through Proposition 3.1 of \cite{chassagneux2014probabilistic} that relates real derivatives to L-derivatives) by applying It\^{o}'s formula to the process
     $$ \Big\{(Y^{i,N}_s - Z^{i,N,h}_s) \big(b(Z^{i,N,h}_s, \eulerlaw[s])- b(Z^{i,N,h}_{t_0}, \eulerlaw[t_0]) \big) \Big\}_{s \geq t_0}, $$ 
     which gives
     \begin{eqnarray}
     && (Y^{i,N}_s - Z^{i,N,h}_s) \big(b(Z^{i,N,h}_s, \eulerlaw[s])- b(Z^{i,N,h}_{t_0}, \eulerlaw[t_0]) \big) \nonumber \\
     & = & \int_{t_0}^s \big(b(Z^{i,N,h}_u, \eulerlaw[u])- b(Z^{i,N,h}_{t_0}, \eulerlaw[t_0]) \big) \, d(Y^{i,N}_u- Z^{i,N,h}_u) \nonumber \\
     && + \sum_{j \neq i} \int_{t_0}^s (Y^{i,N}_u- Z^{i,N,h}_u) \Big( \frac{1}{N} \pmu b (Z^{i,N,h}_u, \eulerlaw[u])(Z^{j,N,h}_u) \Big) \, dZ^{j,N,h}_u \nonumber \\
     && + \int_{t_0}^s (Y^{i,N}_u- Z^{i,N,h}_u) \Big( \frac{1}{N} \pmu b (Z^{i,N,h}_u, \eulerlaw[u])(Z^{i,N,h}_u) + \partial_x b(Z^{i,N,h}_u, \eulerlaw[u]) \Big) \, dZ^{i,N,h}_u \nonumber \\
     && + \frac{1}{2} \sum_{j \neq i} \int_{t_0}^s (Y^{i,N}_u- Z^{i,N,h}_u) \Big( \frac{1}{N} \partial_v \pmu b (Z^{i,N,h}_u, \eulerlaw[u])(Z^{j,N,h}_u) \nonumber \\
     && + \frac{1}{N^2} \ptwomu b(Z^{i,N,h}_u, \eulerlaw[u])( Z^{j,N,h}_u, Z^{j,N,h}_u) \Big) \, d \lev Z^{j,N,h} \rev_u \nonumber \\
     && + \frac{1}{2} \int_{t_0}^s (Y^{i,N}_u- Z^{i,N,h}_u) \Big( \frac{1}{N} \partial_v \pmu b (Z^{i,N,h}_u, \eulerlaw[u])(Z^{i,N,h}_u) \nonumber \\
     && + \frac{1}{N^2} \ptwomu b(Z^{i,N,h}_u, \eulerlaw[u])( Z^{i,N,h}_u, Z^{i,N,h}_u) + \frac{2}{N} \partial_x \pmu b(Z^{i,N,h}_u, \eulerlaw[u])(Z^{i,N,h}_u) \nonumber \\
     && + \partial^2_x b(Z^{i,N,h}_u, \eulerlaw[u]) \Big) \, d \lev Z^{i,N,h} \rev_u. \nonumber 
     \end{eqnarray}
     Putting $t_0 = \eta(s)$, taking average of $i$ from $1$ to $N$, taking expectation and rewriting terms, we have
$$   \frac{1}{N} \sum_{i=1}^N \bE \bigg[ (Y^{i,N}_s - Z^{i,N,h}_s) \big(b(Z^{i,N,h}_s, \eulerlaw[s])- b(Z^{i,N,h}_{\eta(s)}, \eulerlaw[\eta(s)]) \big) \bigg] = \cI_1 + \cI_2,$$ 
where 
$$ \cI_1:= \frac{1}{N} \sum_{i=1}^N \bE \bigg[ \int_{\eta(s)}^s \big(b(Z^{i,N,h}_u, \eulerlaw[u])- b(Z^{i,N,h}_{\eta(s)}, \eulerlaw[\eta(s)]) \big) \big(b(Y^{i,N}_u, \nlaw[u])- b(Z^{i,N,h}_{\eta(s)}, \eulerlaw[\eta(s)]) \big) \,du \bigg] $$ 
and $$ \cI_2:= \frac{1}{N} \sum_{i=1}^N \bE \bigg[ \int_{\eta(s)}^s (Y^{i,N}_u - Z^{i,N,h}_u) \cD^i_u \,du \bigg] , $$ 
where \begin{eqnarray}
\cD^i_u & := & \frac{1}{N} \sum_{j=1}^N \bigg( \pmu b(Z^{i,N,h}_u, \eulerlaw[u])(Z^{j,N,h}_u) b(Z^{j,N,h}_{\eta(u)}, \etaeulerlaw[u]) \bigg) + \px b(Z^{i,N,h}_u, \eulerlaw[u]) b(Z^{i,N,h}_{\eta(u)}, \etaeulerlaw[u]) \nonumber \\
&& + \frac{1}{2} \sigma^2 \sum_{j=1}^N \bigg( \frac{1}{N^2} \ptwomu b (Z^{i,N,h}_u, \eulerlaw[u])(Z^{j,N,h}_u, Z^{j,N,h}_u) + \frac{1}{N} \partial_v \pmu b (Z^{i,N,h}_u,  \eulerlaw[u])(Z^{j,N,h}_u) \bigg) \nonumber \\
&& + \frac{1}{2} \sigma^2 \bigg( \frac{2}{N} \px \pmu b(Z^{i,N,h}_u, \eulerlaw[u])(Z^{i,N,h}_u) + \ptwox b(Z^{i,N,h}_u, \eulerlaw[u]) \bigg). \nonumber 
\end{eqnarray}
By the hypothesis on $b$, all derivatives of $b$ are uniformly bounded. Moreover, by \eqref{eq:Lip}, $b$ has linear growth in space and measure. Therefore,
$$ \frac{1}{N} \sum_{i=1}^N \bE|\cD^i_u|^2 \leq C \Big( 1+\frac{1}{N} \sum_{i=1}^N \bE | Z^{i,N,h}_{\eta(u)}|^2 \Big).  $$ 
Then, by \eqref{eq:boundedinL2},
$$  \sup_{u \in [0,T]}  \Big[ \frac{1}{N} \sum_{i=1}^N \bE|\cD^i_u|^2 \Big] \leq C.  $$ 
By first applying the Cauchy-Schwarz inequality to the expectation operator and then to the sum,
\begin{eqnarray}
\cI_2 & \leq & \int_{\eta(s)}^s \bigg( \frac{1}{N} \sum_{i=1}^N \bE|Y^{i,N}_u - Z^{i,N,h}_u|^2 \bigg)^{1/2} \bigg( \frac{1}{N} \sum_{i=1}^N \bE|\cD^i_u|^2 \bigg)^{1/2} \,du \nonumber \\
& \leq & C \bigg( \sup_{u \in [0,s]} \frac{1}{N} \sum_{i=1}^N \bE|Y^{i,N}_u - Z^{i,N,h}_u|^2 \bigg)^{1/2} h \nonumber\\
& \leq & C \bigg( \frac{1}{2} \sup_{u \in [0,s]} \frac{1}{N} \sum_{i=1}^N \bE|Y^{i,N}_u - Z^{i,N,h}_u|^2 + \frac{1}{2} h^2 \bigg). \label{eq:second term strong error YZ}  
\end{eqnarray}
Next, we rewrite $\cI_1$ as
\begin{eqnarray}
\cI_1 & = & \frac{1}{N} \sum_{i=1}^N \bE \bigg[ \int_{\eta(s)}^s \big(b(Z^{i,N,h}_u, \eulerlaw[u])- b(Z^{i,N,h}_{\eta(s)}, \eulerlaw[\eta(s)]) \big) \big(b(Y^{i,N}_u, \nlaw[u])- b(Z^{i,N,h}_{u}, \eulerlaw[u]) \big) \,du \bigg] \nonumber \\
&& + \frac{1}{N} \sum_{i=1}^N \bE \bigg[ \int_{\eta(s)}^s \big(b(Z^{i,N,h}_u, \eulerlaw[u])- b(Z^{i,N,h}_{\eta(s)}, \eulerlaw[\eta(s)]) \big)^2 \, du \bigg]. \nonumber 
\end{eqnarray}
It is clear that 
\begin{equation}  \frac{1}{N} \sum_{i=1}^N \bE \bigg[ \int_{\eta(s)}^s \big(b(Z^{i,N,h}_u, \eulerlaw[u])- b(Z^{i,N,h}_{\eta(s)}, \eulerlaw[\eta(s)]) \big)^2 \, du \bigg] \leq Ch^2. \label{eq:third term strong error YZ}   \end{equation}  
By the Cauchy-Schwarz inequality and \eqref{eq:Lip}, the first term of $\cI_1$ is bounded by
\begin{eqnarray}
&&  \frac{1}{N} \sum_{i=1}^N \bE \bigg[ \int_{\eta(s)}^s  \big(b(Z^{i,N,h}_u, \eulerlaw[u])- b(Z^{i,N,h}_{\eta(s)}, \eulerlaw[\eta(s)]) \big) \big(b(Y^{i,N}_u, \nlaw[u])- b(Z^{i,N,h}_{u}, \eulerlaw[u]) \big) \,du \bigg]  \nonumber \\
& \leq & \frac{1}{N} \sum_{i=1}^N \int_{\eta(s)}^s \bigg( \bE \bigg|  b(Z^{i,N,h}_u, \eulerlaw[u])- b(Z^{i,N,h}_{\eta(s)}, \eulerlaw[\eta(s)])  \bigg|^2 \bigg)^{1/2} \nonumber \\
&& \, \, \bigg( \bE \bigg|  b(Y^{i,N}_u, \nlaw[u])- b(Z^{i,N,h}_{u}, \eulerlaw[u])\bigg|^2 \bigg)^{1/2} \, du \nonumber \\
& \leq & \frac{1}{N} \sum_{i=1}^N C \sqrt{h} \int_{\eta(s)}^s \bigg( \bE \bigg|  b(Y^{i,N}_u, \nlaw[u])- b(Z^{i,N,h}_{u}, \eulerlaw[u])\bigg|^2 \bigg)^{1/2} \, du \nonumber \\
& \leq & \frac{1}{N} \sum_{i=1}^N C \sqrt{h} \int_{\eta(s)}^s \bigg( \bE|Y^{i,N}_u - Z^{i,N,h}_u |^2 + \frac{1}{N} \sum_{j=1}^N  \bE|Y^{j,N}_u - Z^{j,N,h}_u |^2 \bigg)^{1/2} \,du  \nonumber \\
& \leq & \frac{2}{N} \sum_{i=1}^N C \sqrt{h} \int_{\eta(s)}^s \big( \bE|Y^{i,N}_u - Z^{i,N,h}_u |^2 \big)^{1/2} \,du \nonumber \\
& \leq & 2C h^{3/2} \bigg[ \sup_{u \in [0,s]} \bigg( \frac{1}{N} \sum_{i=1}^N  \bE|Y^{i,N}_u - Z^{i,N,h}_u |^2 \bigg) \bigg]^{1/2} \nonumber \\
& \leq & C \bigg( h^{3} + \sup_{u \in [0,s]} \bigg( \frac{1}{N} \sum_{i=1}^N  \bE|Y^{i,N}_u - Z^{i,N,h}_u |^2 \bigg) \bigg). \label{eq:fourth term strong error YZ}    \end{eqnarray}
A combination of \eqref{eq: ito strong error YZ}, \eqref{eq:first term strong error YZ}, \eqref{eq:second term strong error YZ}, \eqref{eq:third term strong error YZ}  and \eqref{eq:fourth term strong error YZ}  gives
$$  \sup_{u \in [0,t]} \bigg[ \frac{1}{N} \sum_{i=1}^N \bE (Y^{i,N}_{u} - Z^{i,N,h}_{u})^2 \bigg] \leq C \bigg( \int_0^t \sup_{u \in [0,s]} \bigg[  \frac{1}{N} \sum_{i=1}^N \bE |Y^{i,N}_u - Z^{i,N,h}_u|^2 \bigg] \,ds + h^2 \bigg),  \quad \forall t \in [0,T],$$ 
which implies by Gronwall's inequality that
$$  \sup_{u \in [0,T]} \bigg[ \frac{1}{N} \sum_{i=1}^N \bE (Y^{i,N}_{u} - Z^{i,N,h}_{u})^2 \bigg] \leq Ch^2. $$ 
Since the constant $C$ does not depend on $N$, we conclude that
$$ \sup_{N \in \bN} \sup_{s \in [0,T]}\bE \big[ W_2(\nlaw[s], \eulerlaw[s])^2]  \leq \sup_{N \in \bN} \sup_{s \in [0,T]} \bigg[ \frac{1}{N} \sum_{i=1}^N \bE (Y^{i,N}_{s} - Z^{i,N,h}_{s})^2 \bigg] \leq Ch^2. $$ 
    \end{proof}
    A combination of Lemma \ref{YZ strong bound W2}  and Theorem \ref{variance antithetic non-discretised} immediately gives the following result.
    \begin{theorem}[Variance of antithetic difference] \label{variance antithetic discretised}
    Assume \eqref{eq:Int}. Suppose that $b \in \cM_4 \big( \bR^d  \times \mathcal{P}_2(\bR^d) \big) $ and $\Phi \in \cM_4 \big(  \mathcal{P}_2(\bR^d) \big) $. Moreover, suppose that $\sigma$ is constant.  Then 
    \[ \text{Var} \Big[ \Phi (\eulernlawellh{T}{}{h}) -  \frac{1}{2} \big( \Phi (\eulernlawoneellh{T}{}{2h}) + \Phi (\eulernlawtwoellh{T}{}{2h}) \big) \Big] \leq  C \Big( \frac{1}{N^2} + h^2 \Big), 
    \]
    where $C$ is a constant that depends on $\Phi$, $b$, $\sigma$ and $T$, but does not depend on $N$ or $h$.
    \end{theorem}
    
      For an estimator involving an Euler numerical scheme with discretisation step $h$, its \emph{computational complexity} is defined by     $$ \text{Computational complexity} := h^{-1}  \Big( \text{Order of interactions of estimator} \Big) . $$
    As before, we perform an analysis on the complexity of this algorithm  by assuming that $b$ is of the form \eqref{pth order b} and that $\sigma$ is constant. By Theorem \ref{weak error Mckean euler}, since $h_{\ell}= \frac{T}{N_{\ell}}$, 
 \[
| \bE[\Phi(\eulernlawellh{T}{\ell}{h_{\ell}})] - \Phi(\mclaw[T]) | \leq \frac{C}{N_{\ell}}. \label{eq:I} \tag{I}
\]
Moreover, by Theorem \ref{variance antithetic discretised}, we have
\[
\text{Var} \Big[ \Phi(\eulernlawelli{T}{\ell}{\theta}{\ell}{h_{\ell}}) - \frac{1}{2} \Big( \Phi(\eulernlawoneelli{T}{\ell}{\theta}{\ell}{2h_{\ell}}) + \Phi(\eulernlawtwoelli{T}{\ell}{\theta}{\ell}{2h_{\ell}}) \Big) \Big] \leq \frac{C}{N^2_{\ell}}. \label{eq:II} \tag{II}
\]Finally, by Definition \ref{def interactions complexity}, the complexity of the antithetic difference is bounded by
\[
\text{Complexity} \Big[ \Phi(\eulernlawelli{T}{\ell}{\theta}{\ell}{h_{\ell}}) - \frac{1}{2} \Big( \Phi(\eulernlawoneelli{T}{\ell}{\theta}{\ell}{2h_{\ell}}) + \Phi(\eulernlawtwoelli{T}{\ell}{\theta}{\ell}{2h_{\ell}}) \Big) \Big]  \leq CN_{\ell}^{p+2}. \label{eq:III} \tag{III}
\]
\begin{theorem}[Complexity of antithetic MLMC with time discretisation for estimator \eqref{eq: MLMC def }] \label{MLMC time discret complexity}
Assume \eqref{eq:Int}. Suppose that $b$ is of the form \eqref{pth order b}. Furthermore, suppose that $b \in \cM_4 \big( \bR^d  \times \mathcal{P}_2(\bR^d) \big) $, $\Phi \in \cM_4 \big(  \mathcal{P}_2(\bR^d) \big) $ and $\sigma$ is constant.  Then there exist constants $C_1, C_2>0$ such that for any $\eps < e^{-1}$, there exist a value $L$ and a sequence $\{ M_{\ell} \}_{\ell=0}^L$ such that the mean-square error of $ {\cA}^{\text{A- MLMC,$t$}}$ is bounded by
$$  \bE \big[ \big(  {\cA}^{\text{A- MLMC,$t$}} - \Phi(\mclaw[T]) \big)^2 \big]  \leq C_1 \eps^2 $$ 
and the complexity of $ {\cA}^{\text{A- MLMC,$t$}}$  is bounded by 
\[ \text{\emph{Complexity}} \, \big(  {\cA}^{\text{A- MLMC,$t$}}\big) \leq  
      C_2 \eps^{-2-p}. 
\]
\end{theorem}
\begin{proof}
As in Theorem \ref{complexity no time discret},  the proof of this theorem is also almost identical to the proof of Theorem 1 in \cite{cliffe2011multilevel}. Nonetheless, we present the proof with explicit expressions for $L$ and $\{ M_{\ell} \}_{\ell=0}^L$ so that practitioners can implement this algorithm easily. Set $$ L:= \ceil{\log_2(\sqrt{2} \eps^{-1})}, \quad \quad M_{\ell}:= \ceil{ 2 \eps^{-2} 2^{pL/2} (1-2^{-p/2})^{-1} 2^{-(p+4) \ell/2}}, \, \, \ell \in \{0, \ldots, L \}. $$ 
By standard decomposition of the mean-square error, we have  
$$ \text{Mean-square error} =  \text{Var}  ({\cA}^{\text{A- MLMC,$t$}}) +  (\bE ({\cA}^{\text{A- MLMC,$t$}}) - \Phi(\mclaw[T]))^2.  $$ 
By the choice of $L$, $ 2^{-L} \leq \frac{\eps}{\sqrt{2}}.$ Therefore, by Property \eqref{eq:I},
\begin{equation} | \bE({\cA}^{\text{A- MLMC,$t$}}) - \Phi(\mclaw[T])|^2 = |\bE[\Phi(\eulernlawellh{T}{L}{h_{L}})] - \Phi(\mclaw[T])|^2 \leq \Big( \frac{C}{N_{L}} \Big)^2 = (C 2^{-L})^2 \leq C^2 \Big( \frac{\eps^2}{2} \Big). \label{eq:weak error bound MLMC A}   \end{equation}
On the other hand, by Property \eqref{eq:II} and the choice of $\{ M_{\ell} \}_{\ell=0}^L$, 
\begin{eqnarray}
\text{Var} ({\cA}^{\text{A- MLMC,$t$}})  \leq  \sum_{\ell=0}^L \frac{1}{M^2_{\ell}} \bigg[  \sum_{\theta=1}^{M_{\ell}} \frac{C}{N^2_{\ell}} \bigg] \leq \sum_{\ell=0}^L \frac{C}{M_{\ell}} 2^{-2\ell} & \leq & \sum_{\ell=0}^L C2^{-2\ell} \Big( 2^{-1} \eps^2 2^{-pL/2} (1-2^{-p/2}) 2^{(p+4) \ell/2} \Big)  \nonumber \\
& = & C 2^{-1} \eps^2 2^{-pL/2}(1-2^{-p/2}) \sum_{\ell=0}^L 2^{p\ell/2} \nonumber \\
& < & \frac{1}{2} C \eps^2. \nonumber 
\end{eqnarray}
This verifies that the mean-square error is bounded by $\frac{1}{2} (C^2+C) \eps^2$. Next, we note that
$$ M_{\ell} \leq 2 \eps^{-2} 2^{pL/2} (1-2^{-p/2})^{-1} 2^{-(p+4) \ell/2} + 1$$
and hence, by Property \eqref{eq:III}, 
\begin{equation} \text{Complexity} ({\cA}^{\text{A- MLMC,$t$}}) \leq C \bigg( \sum_{\ell=0}^L 2 \eps^{-2} 2^{pL/2} (1-2^{-p/2})^{-1} 2^{-(p+4) \ell/2} 2^{(p+2)\ell} + \sum_{\ell=0}^L 2^{(p+2) \ell} \bigg).  \label{eq:cost A two parts} \end{equation}
Note that the choice of $L$ implies that $2^L \leq 2 \sqrt{2} \eps^{-1}.$ 
\begin{eqnarray}
\sum_{\ell=0}^L 2 \eps^{-2} 2^{pL/2} (1-2^{-p/2})^{-1} 2^{-(p+4) \ell/2} 2^{(p+2) \ell} & = & 2 \eps^{-2} 2^{pL/2} (1-2^{-p/2})^{-1} \sum_{\ell=0}^L 2^{p\ell/2} \nonumber \\
& < & 2 \eps^{-2} 2^{pL/2} (1-2^{-p/2})^{-1} \Big( 2^{pL/2} (1-2^{-p/2})^{-1} \Big) \nonumber \\
& = & 2 \eps^{-2} 2^{pL} (1-2^{-p/2})^{-2} \nonumber \\
& \leq & 2 \big( 2 \sqrt{2} \big)^p (1-2^{-p/2})^{-2} \eps^{-2-p}. \label{eq:cost A first part}
\end{eqnarray}
Similarly,
\begin{equation}
    \sum_{\ell=0}^L 2^{(p+2) \ell} \leq \frac{2^{(p+2)L}}{1-2^{-(p+2)}} \leq \frac{(2 \sqrt{2})^{p+2}}{1-2^{-(p+2)}} \eps^{-(p+2)}. \label{eq:cost A second part}
\end{equation}
A combination of \eqref{eq:cost A two parts}, \eqref{eq:cost A first part} and \eqref{eq:cost A second part} finally gives
$$ \text{Complexity} ({\cA}^{\text{A- MLMC,$t$}}) \leq C \bigg( 2 \big( 2 \sqrt{2} \big)^p (1-2^{-p/2})^{-2} + \frac{(2 \sqrt{2})^{p+2}}{1-2^{-(p+2)}} \bigg) \eps^{-2-p}.$$ 
\end{proof}
\begin{subappendices}
\renewcommand{\thesection}{A}%
\section{Appendix: A review of linear functional derivatives and L-derivatives} \label{ linear functional derivatives and L-derivatives appendix} 
    
  Our method of proof is based on the theory of calculus on   the Wasserstein space. A substantial portion of the appendix is extracted from a recent work \cite{chassagneux2019weak}. We make an intensive use of the so-called ``L-derivatives'' and ``linear functional derivatives'' that we recall now, following essentially \cite{cardaliaguet2015master}. We also introduce higher-order versions of these derivatives as they are needed in the proofs. \\
${} $ \\
\textbf{Linear functional derivatives}\label{sec lfd}  ${} $ \\ 

A continuous function $\frac{\delta U}{\delta m}: \cP_2(\bR^d) \times \bR^d \to \bR$ is said to be the \emph{linear functional derivative} of $U: \cP_2(\bR^d) \to \bR$, if 
\begin{itemize}
\item for any bounded set $\cK \subset \cP_2(\R^d)$, $y \mapsto \frac{\delta U}{\delta m}(m,y)$ has at most quadratic growth in $y$ uniformly in $m \in \cK$,
\item for any $m, m' \in \cP_2(\bR^d)$,
 \begin{align}\label{eq de first order deriv}
 U(m')- U(m) = \int_0^1 \int_{\bR^d} \frac{\delta U}{\delta m}( (1-s)m + sm',y) \, (m'-m)(dy) \, ds. 
 \end{align}
\end{itemize}

\noindent For the purpose of our work, we need to introduce derivatives at any order $p\ge 1$. 



\begin{definition}
For any $p \ge 1$, the $p$-th order linear functional of the function $U$ is a continuous
function from $\frac{\delta^p U}{\delta m^p}: \cP_2(\R^d)\times (\R^d)^{p-1}\times \R^d \rightarrow \R$ satisfying
\begin{itemize}
\item for any bounded set $\cK \subset \cP_2(\R^d)$, $(y,y') \mapsto \frac{\delta^p U}{\delta m^p}(m,y,y')$ has at most quadratic growth in $(y,y')$ uniformly in $m \in \cK$,
\item for any $m, m' \in \cP_2(\bR^d)$,
\begin{align*}
 \frac{\delta^{p-1} U}{\delta m^{p-1}}(m',y) - \frac{\delta^{p-1} U}{\delta m^{p-1}}(m,y) = \int_0^1 \int_{\bR^{d}} \frac{\delta^p U}{\delta m^p}((1-s)m +sm',y,y') \, (m'-m)( \ud y') \,\ud s,
\end{align*}
provided that the $(p-1)$-th order derivative is well defined.
\end{itemize}
\end{definition}
{
\noindent The above derivatives are defined up to an additive constant via \eqref{eq de first order deriv}.
They are normalised by
\begin{equation}
   \frac{\delta^p U}{\delta m^p}(m,y_1, \ldots, y_p) = 0, \quad \text{ if } y_i=0 \, \, \text{ for some } i \in \{1, \ldots, p \}.
    \label{eq: normalisation linear functional deriatives}
\end{equation} }

 ${} $ \\
\textbf{L-derivatives}\label{sec lion d}  ${} $ \\ 


The above notion of linear functional derivatives is not enough for our work. We shall need to consider further derivatives in the non-measure argument of the derivative function.

If the function $y \mapsto \frac{\delta U}{\delta m}(m,y)$ is of class $\cC^1$, we consider the \emph{intrinsic} derivative of $U$ that we denote
\begin{align*}
\partial_\mu U(m,y) := \partial_y \frac{\delta U}{\delta m}(m,y)\;.
\end{align*}
The notation is borrowed from the literature on mean field games and corresponds to the notion of ``L-derivative'' introduced by P.-L. Lions in his lectures at Coll\`{e}ge de France \cite{lions2014cours}. Traditionally, it is introduced by considering a lift on an $L^2$ space of the function $U$ and using the Fr\'{e}chet differentiability of this lift on this Hilbert space. The equivalence between the two notions is proved in \cite[Tome I, Chapter 5]{carmona2017probabilistic}, where the link with the notion of derivatives used in optimal transport theory is also made.

In this context, higher order derivatives are introduced by iterating the operator $\partial_\mu$ and the derivation in the non-measure arguments. Namely, at order $2$, one considers
\begin{align*}
\cP_2(\R^d)\times \R^d \ni   (m,y) \mapsto \partial_y \partial_\mu U(m,y) \text{ and } \
\cP_2(\R^d)\times \R^d\times \R^d \ni   (m,y,y') \mapsto  \partial^2_{\mu} U(m,y,y') \;.
\end{align*}

\noindent Inspired by the work \cite{crisan2017smoothing}, for any $k \in \bN$, we formally define  the  higher order derivatives in measures through the following iteration (provided that they actually exist): for any $k \geq 2$, $(i_1, \ldots, i_k) \in \{ 1, \ldots, d \}^k$ and $x_1, \ldots, x_k \in \bR^d$, the function $\partial^k_{\mu} f:\mathcal{P}_2(\bR^d) \times (\bR^d)^{ k} \to (\bR^d)^{\otimes k}$ is defined by
\begin{equation} \bigg( \partial^k_{\mu} f( \mu, x_1, \ldots, x_k) \bigg)_{(i_1, \ldots, i_k)} := \bigg( \pmu \bigg( \Big( \partial^{k-1}_{\mu} f( \cdot, x_{1}, \ldots, x_{k-1}) \Big)_{(i_1, \ldots, i_{k-1})} \bigg)(\mu,x_k) \bigg)_{i_k}, \label{eq:generalformulaintro} \end{equation} 
and its corresponding mixed derivatives in space $\partial^{\ell_k}_{v_k} \ldots \partial^{\ell_1}_{v_1} \partial^k_{\mu} f:\mathcal{P}_2(\bR^d) \times (\bR^d)^{ k} \to (\bR^d)^{\otimes (k+ \ell_1 +\ldots \ell_k)} $ are defined by
\begin{equation} \bigg( \partial^{\ell_k}_{v_k} \ldots \partial^{\ell_1}_{v_1} \partial^k_{\mu} f( \mu, x_1, \ldots, x_k) \bigg)_{(i_1, \ldots, i_k)} := \frac{\partial^{\ell_k}}{\partial x^{\ell_k}_k} \ldots \frac{\partial^{\ell_1}}{\partial x^{\ell_1}_1} \bigg[ \bigg( \partial^k_{\mu} f( \mu, x_1, \ldots, x_k) \bigg)_{(i_1, \ldots, i_k)} \bigg], \quad \ell_1 \ldots \ell_k \in \bN \cup \{0 \}. \label{eq:generalformulamixed} \end{equation}
Since this notation for higher order derivatives in measure is quite cumbersome, we introduce the following multi-index notation for brevity. This notation was first proposed in \cite{crisan2017smoothing}. 
\begin{definition}[Multi-index notation]
Let $ n, \ell$ be non-negative integers. Also, let $\bm{\beta}=(\beta_1, \ldots, \beta_n)$ be an $n$-dimensional vector of non-negative integers. Then we call any ordered tuple of the form $(n,\ell, \bm{\beta})$ or $(n,\bm{\beta})$ a \emph{multi-index}. For a function $f:\R^d \times\cP_2(\R^d) \to \R$, the derivative $D^{(n,\ell, \bm{\beta})} f(x, \mu , v_1, \ldots, v_n)$  is defined as 
$$D^{(n,\ell, \bm{\beta})} f (x, \mu , v_1, \ldots, v_n) :=  \partial^{\beta_n}_{v_n} \ldots \partial^{\beta_1}_{v_1} \partial^{\ell}_x \pnmu f( x, \mu , v_1, \ldots, v_n)$$
if this derivative is well-defined.
For any function $\Phi: \mathcal{P}_2 ( \bR^d) \to \bR$, we define
$$D^{(n,\bm{\beta})} \Phi( \mu , v_1, \ldots, v_n) :=  \partial^{\beta_n}_{v_n} \ldots \partial^{\beta_1}_{v_1} \pnmu \Phi( \mu , v_1, \ldots, v_n), $$ 
if this derivative is well-defined. Finally, we also define the \emph{order} \footnote{ We do not consider `zeroth' order derivatives in our definition, i.e. at least one of $n$, $\beta_1, \ldots, \beta_n$ and $\ell$ must be non-zero, for every multi-index $\big(n, \ell, (\beta_1, \ldots, \beta_n) \big)$.} $ |(n,\ell, \bm{\beta})| $ (resp.  $|(n,\bm{\beta})|$ ) by
\begin{equation} |(n,\ell, \bm{\beta})|:= n+ \beta_1 + \ldots \beta_n + \ell , \quad \quad |(n,\bm{\beta})|:= n+ \beta_1 + \ldots \beta_n .  \label{eq:orderdef}
\end{equation}
\end{definition}

In our proofs, we aim to formulate sufficient conditions purely in terms of regularity of the drift and diffusion functions, as well as the test function.  A class $\cM_k$ of regularity in differentiating measures is proposed. 

\begin{definition}[Class $\cM_k$ of $k$th order differentiable functions] ${}$ \label{eq:classmk} \\
\begin{enumerate}[(i)]
\item \label{Mk b and sigma} The functions $b$ and $\sigma$ belong to class $\cM_k( \bR^d \times \cP_2 ( \bR^d)) $, if the derivatives  $D^{(n,\ell, \bm{\beta})} b (x,\mu,v_1, \ldots, v_n)$ and $D^{(n,\ell, \bm{\beta})} \sigma (x,\mu, v_1, \ldots, v_n)$ exist for every multi-index $(n,\ell, \bm{\beta})$ such that $|(n,\ell, \bm{\beta})| \leq k$ and 
\begin{enumerate}[(a)]
\item \begin{equation}   \big| D^{(n,\ell, \bm{\beta})} b (x,\mu, v_1, \ldots, v_n) \big|  
 \leq  C, \quad \quad  \big| D^{(n,\ell, \bm{\beta})}  \sigma (x,\mu, v_1, \ldots, v_n) \big| \leq C, \label{eq:boundbandsigma} \end{equation}
\item \begin{eqnarray} && \Big| D^{(n,\ell, \bm{\beta})} b (x,\mu, v_1, \ldots, v_n) -D^{(n,\ell, \bm{\beta})} b  (x', \mu' ,v'_1, \ldots, v'_n)  \Big| \nonumber \\
& \leq & C \bigg( |x-x'| + \sum_{i=1}^n|v_i-v'_i| + W_2 (\mu,\mu') \bigg), \nonumber  \\
&&  \Big| D^{(n,\ell, \bm{\beta})} \sigma (x,\mu, v_1, \ldots, v_n) -D^{(n,\ell, \bm{\beta})} \sigma  (x', \mu' ,v'_1, \ldots, v'_n)  \Big| \nonumber \\
& \leq & C \bigg( |x-x'| + \sum_{i=1}^n|v_i-v'_i| + W_2 (\mu,\mu') \bigg),  \label{eq:Lipbandsigma}
\end{eqnarray}
\end{enumerate}
for any $x,x',v_1,v'_1,\ldots, v_n, v'_n \in \bR^d$ and $\mu, \mu' \in \mathcal{P}_2(\bR^d)$, for some constant $C>0$.
\item  Any function $\Phi : \mathcal{P}_2 ( \bR^d) \to \bR$ is said to be in $\cM_k( \cP_2 ( \bR^d)) $, if $D^{(n, \bm{\beta})} \Phi (\mu,v_1, \ldots, v_n)$ exists for every  multi-index $(n, \bm{\beta})$ such that $|(n, \bm{\beta})| \leq k$ and \begin{enumerate}[(a)]
\item \begin{equation}  \Big| D^{(n, \bm{\beta})} \Phi (\mu, v_1, \ldots, v_n) \Big|  \leq C,  \label{eq:boundf} \end{equation}
\item \begin{eqnarray} &&   \Big| D^{(n, \bm{\beta})} \Phi(\mu, v_1, \ldots, v_n) -D^{(n, \bm{\beta})} \Phi( \mu' ,v'_1, \ldots, v'_n)  \Big|  \nonumber \\
& \leq &  C \bigg(  \sum_{i=1}^n|v_i-v'_i| + W_2 (\mu,\mu') \bigg),  \label{eq:Lipf}
\end{eqnarray}
\end{enumerate}
for any $v_1,v'_1,\ldots, v_n, v'_n \in \bR^d$ and $\mu, \mu' \in \mathcal{P}_2(\bR^d)$, for some constant $C>0$.
\item A function $\cV: [0,T] \times   \mathcal{P}_2 ( \bR^d) \to \bR$ is said  to be in $\cM_k \big( [0, T]  \times \mathcal{P}_2(\bR^d) \big) $, if $\cV(\cdot,  \mu)$ is in $C^1([0,T])$, for each $\mu \in \mathcal{P}_2 ( \bR^d)$  and $\cV(s,  \cdot) \in \cM_k \big(   \mathcal{P}_2(\bR^d) \big) $, for each $s \in [0,T]$, where the $L^{\infty}$ and Lipschitz bounds of the derivatives of $\cV(s, \cdot)$ are uniform in time, i.e. they only depend on $T$.
\end{enumerate} 
\end{definition}
{
As for the first order case, we can establish the following relationship with linear functional derivatives, see e.g. \cite{cardaliaguet2015master} for the correspondence up to order 2,
\begin{equation}
   \partial^n_\mu U(\cdot) =  \partial_{y_n} \frac{\delta }{\delta m} \dots \partial_{y_1} \frac{\delta }{\delta m} U(\cdot) = \partial_{y_n}\dots \partial_{y_1}\frac{\delta^n }{\delta m^n} U(\cdot) \,, \label{eq: connection lions functional derivatives high order} 
\end{equation}
provided one of the two derivatives is well-defined.
}
The following proposition (Lemma 2.5 from \cite{chassagneux2019weak}) relates regularity of L-derivatives with that of linear functional derivatives. We first define class $\cM^{L}_k$ that characterises $k$th order linear functional derivatives. 
\begin{definition}[Class $\cM^{L}_k$ of $k$th order differentiable functions in linear functional derivatives] \label{eq: def MLk} 
A function $U: \cP_2(\bR^d) \to \bR$ is said to be in class $\cM^{L}_k( \cP_2 (\bR^d))$ if it is $k$ times differentiable in the sense of linear functional derivatives and satisfies 
\begin{equation} \bigg| \frac{\delta^k U}{\delta m^k} (m,y_1, \ldots, y_k) \bigg| \leq C \big( |y_1|^k + \ldots |y_k|^k \big), \label{eq: first order lions linear functional def} \end{equation}
for some constant $C>0$ that does not depend on $m$ and $y_1, \ldots, y_k$. 
\end{definition}
\begin{proposition}[Lemma 2.5 from \cite{chassagneux2019weak}] \label{theorem lions linear functional} 
Suppose that $U \in \cM_k ( \cP_2 (\bR^d))$. Then
\begin{equation} \bigg| \frac{\delta^k U}{\delta m^k} (m,y_1, \ldots, y_k) \bigg| \leq \frac{(\sqrt{d})^k}{k} \| \partial^k_{\mu} U \|_{\infty} \big( |y_1|^k + \ldots |y_k|^k \big). \label{eq: first order lions linear functional} \end{equation}
Consequently, $U \in \cM^{L}_k( \cP_2 (\bR^d))$.
\end{proposition}

\renewcommand{\thesection}{B}%
\section{Appendix: Weak error analysis} \label{appendix weak error analysis} 
In this section, we consider the following weak errors of the form 
\[
\Big| \Phi (\mclaw[T]) - \bE[\Phi(\nlaw[T])] \Big| \quad \quad \text{ and } \quad \quad \Big| \Phi (\mclaw[T]) - \bE[\Phi(\eulerlaw[T])] \Big|, 
\] 
for functionals $\Phi: \mathcal{P}_2( \bR^d) \to \bR$.  The method of analysis follows from the work \cite{chassagneux2019weak}.
For any square-integrable random variable $\eta$, we define  
\begin{equation} \label{eq:McKeanflow1}
    X^{s, \eta}_t = \eta + \int_s^t  b(X^{s, \eta}_r,   \rvlaw[{X^{s,  \eta}_r}]  )  \,dr + \int_s^t \sigma (X^{s, \eta}_r,   \rvlaw[{X^{s, \eta}_r}] )  \,dW_r, \quad t \in [s,T].
\end{equation}
A starting point of our investigation is the Feynman-Kac theorem for functionals of measures established in Theorem 7.2 of \cite{buckdahn2017mean} (for the case $k=2$). The generalisation to $k>2$ is done in Theorem 2.15 of \cite{chassagneux2019weak}. Note that the condition $\cM_1(\bR^d \times \mathcal{P}_2 (\bR^d) ) $ automatically implies \eqref{eq:Lip}.
\begin{theorem} \label{eq:generalisationmainresult}
Let $k \geq 2$ be an integer. Suppose that  $b, \sigma \in \cM_k(\bR^d \times \mathcal{P}_2 (\bR^d) ) $.
  We consider a function $\cV : [0,T] \times \mathcal{P}_2 (\bR^d) \to \bR $ defined by 
\begin{equation}  \cV( s, \rvlaw[\eta]) = \Phi \big( \rvlaw[{X^{s, \eta}_T}] \big) ,  \label{eq: defofflow}
\end{equation}
  for some function $\Phi:  \cP_2(\bR^d)   \to \bR$ in $\cM_k( \cP_2(\bR^d) )$. Then  $\cV \in \cM_k ( [0,T] \times \mathcal{P}_2 (\bR^d)) $ and satisfies the PDE
\begin{equation}  \label{eq pde measure} \begin{cases} 
       \partial_s \cV(s, \mu) + \int_{\bR^d}  \big[ \partial_{\mu} \cV (s, \mu) (x) b( x, \mu)  + \frac{1}{2} \text{Tr} \big( \partial_v \partial_{\mu} \cV  ( s,\mu) ( x) a( x, \mu) \big) \big] \, \mu (dx) =0, & s \in (0,T), \\
      &  \\
      \cV(T,\mu) = \Phi ( \mu), &    \end{cases}
\end{equation} 
where $a=(a_{i,k})_{1 \leq i,k \leq  d }: \bR^d \times \mathcal{P}_2(\bR^d) \to \bR^d \otimes \bR^d$ denotes the diffusion operator
$$ a_{i,k} (x,\mu) := \sum_{j=1}^m  \sigma_{i,j} (x,\mu) \sigma_{k,j} (x,\mu), \quad \quad \forall x \in \bR^d, \quad \forall \mu \in \mathcal{P}_2(\bR^d).$$
\end{theorem}

We make the following observations before starting the main proof. The finite dimensional projection  $V:[0,T]\times(\bR^d)^N \to \bR$ is defined by
\begin{equation}
    V(s,x_1,\ldots,x_N):= \cV \bigg( s, \frac{1}{N}\sum_{i=1}^N\delta_{x_i} \bigg). \label{eq:finitedimproj}
\end{equation} Proposition 3.1 of \cite{chassagneux2014probabilistic} allows us to conclude that  $V$ is differentiable in the time component and twice-differentiable in the space components. Hence it is legitimate to apply the classical It\^{o}'s formula to $V$.

 Next, by the flow property of  \eqref{eq:McKeanflow1} (see equation (3.5) in \cite{buckdahn2017mean}), we observe that for any $s \in [0,T]$,
\[
\cV(s,\rvlaw[X^{0, \xi}_s]) = \Phi \Big( \rvlaw[X^{s, {X^{0, \xi}_s}}_T] \Big)   = \Phi (\rvlaw[X^{0, \xi}_T])  . 
\]
Hence, this function is constant in time $s \in [0,T]$. In particular, by the terminal condition, we have $$
\Phi \big( \mclaw[T] \big) = \Phi \big( \rvlaw[X^{0, \xi}_T] \big)=  \cV(T, \rvlaw[X^{0, \xi}_T]) = \cV(0, \rvlaw[\xi] ) = \cV(0, \nu ).$$
 By the terminal condition for the PDE, we notice that $$ \Phi( \nlaw[T]) = \cV(T,\nlaw[T]).$$
Therefore, the error between the particle system and the McKean-Vlasov limit decomposes as 
\begin{eqnarray}
\Phi(\nlaw[T])  - \Phi ( \mclaw[T] ) & = &  \cV(T,\nlaw[T]) - \cV(0,\nu) \nonumber \\
& = & \big( \cV(T,\nlaw[T]) - \cV(0,\nlaw[0]) \big)+ \big( \cV(0,\nlaw[0]) - \cV(0,\nu) \big). \label{eq:errordecomp}
\end{eqnarray}
This decomposition enables us to prove the following result. 
\begin{theorem} \label{weak error Mckean particles}
Suppose that $b, \sigma \in \cM_2 \big( \bR^d  \times \mathcal{P}_2(\bR^d) \big) $ and $\Phi \in \cM_2 \big(  \mathcal{P}_2(\bR^d) \big) $. Then the weak error in the particle approximation satisfies
\begin{eqnarray}
&& \Big| \bE[\Phi(\nlaw[T])] - \Phi (\mclaw[T])   \Big|
 \leq \frac{C}{N},  \label{eq:expansionformula} 
\end{eqnarray}
where $C$ is a constant that depends on $\Phi$, $b$, $\sigma$ and $T$, but does not depend on $N$. 
\end{theorem}
\begin{proof}
We first recall the definition of $V$ defined in \eqref{eq:finitedimproj}. By the assumptions on $b$ and $\sigma$, the standard It\^{o}'s formula is applicable to $V$ by Proposition 3.1 of \cite{chassagneux2014probabilistic}. 
Let $\mathbf{x}=(x_1, \ldots, x_N)$. Moreover, we know from this theorem that
\[
\frac{ \partial V}{\partial x_i} ( s,\mathbf{x}) =  \frac{1}{N} \pmu {\cV} \bigg(s, \frac{1}{N}
\sum_{j=1}^N \delta_{x_j} \bigg)(x_i) \]
and
\[
\frac{ \partial^2 V}{\partial x^2_i} ( s,\mathbf{x}) =  \frac{1}{N} \partial_v \pmu {\cV} \bigg(s, \frac{1}{N}
\sum_{j=1}^N \delta_{x_j} \bigg)(x_i) + \frac{1}{N^2} \ptwomu {\cV} \bigg(s, \frac{1}{N}
\sum_{j=1}^N \delta_{x_j} \bigg)(x_i,x_i), \]
for any $s \in [0,T]$, $x_1, \ldots, x_N \in \bR^d$.
Let $\mathbf{Y}^N:=(Y^{1,N} , \ldots, Y^{N,N})$.  Then
\begin{equation*}
\begin{split}
&  {\cV}(T,\nlaw[T]) - {\cV}(0,\nlaw[0])   
 =    V({T}, \mathbf{Y}^N_{T}) - V(0, \mathbf{Y}^N_0)  \\
& =  \bigg[ \int_0^{T} \frac{ \partial V}{ \partial s} (s, \mathbf{Y}^N_s) + \sum_{i=1}^N \frac{\partial V}{\partial x_i} ( s, \mathbf{Y}^N_s)  b \big( Y^{i,N}_s, \nlaw[s] \big)  
+ \frac{1}{2}  \text{Tr} \bigg( a \big( Y^{i,N}_s, \nlaw[s]  \big)   \sum_{i=1}^N \frac{\partial^2 V}{\partial x^2_i}  (s,\mathbf{Y}^N_s) \bigg) \,ds \bigg]  \\
& +  \sum_{i=1}^N \int_0^{T}  \sigma( Y^{i,N}_s, \nlaw[s] )^T \frac{\partial V}{\partial x_i} ( s,\mathbf{Y}^N_s) \cdot dW_s^i  \\
& =   \int_0^{T}  \partial_s {\cV} \big( s, \nlaw[s]  \big) + \sum_{i=1}^N \Bigg[ \frac{1}{N} \partial_{\mu} {\cV} \big( s, \nlaw[s]  \big)( {Y^{i,N}_s} ) b \big( Y^{i,N}_s, \nlaw[s]  \big)   \\
& \, \, + \frac{1}{2}  \text{Tr} \Bigg( a \big( Y^{i,N}_s, \nlaw[s]  \big) \bigg( \frac{1}{N} \partial_{v}   \partial_{\mu} {\cV} \big( s, \nlaw[s]  \big)  (Y^{i,N}_s)  + \frac{1}{N^2} \partial^2_{\mu}  {\cV} \big( s, \nlaw[s] \big)  (Y^{i,N}_s, Y^{i,N}_s) \bigg) \Bigg) \Bigg] \,ds \\
& +  \frac{1}{N} \sum_{i=1}^N \int_0^{T}  \sigma( Y^{i,N}_s, \nlaw[s] )^T \partial_{\mu} {\cV} \big( s, \nlaw[s]  \big) (  Y^{i,N}_s) \cdot dW_s^i.
\end{split}
\end{equation*}
By \eqref{eq:errordecomp} and PDE \eqref{eq pde measure} evaluated at  $(s,\nlaw[s])_{s \in [0,T]}$, the expression simplifies to
\begin{eqnarray} 
     \Phi (\nlaw[T])  - \Phi (\mclaw[T])  & =  & \big( \cV(0,\nlaw[0]) - \cV(0,\nu) \big)  \nonumber \\
     && + \int_0^{T} \frac{1}{2} \Bigg[\frac{1}{N^2} \sum_{i=1}^N  \text{Tr} \bigg( a \big(Y^{i,N}_s, \nlaw[s]  \big)    \partial^2_{\mu}  {\cV} \big( s, \nlaw[s]  \big)  (Y^{i,N}_s, Y^{i,N}_s) \bigg) \Bigg] \,ds \nonumber  \\
&& +\frac{1}{N}\sum_{i=1}^N   \int_0^{T}  \sigma( Y^{i,N}_s, \nlaw[s] )^T \partial_{\mu} {\cV} \big( s, \nlaw[s]  \big) ( Y^{i,N}_s) \cdot dW_s^i. \label{eq representation}
\end{eqnarray}
It follows from Lemma 2.5 and Theorem 2.11 from \cite{chassagneux2019weak} that
$$ \big| \bE\big( \cV(0,\nlaw[0]) - \cV(0,\nu) \big) \big| \leq \frac{C}{N}. $$ 
Taking expectation on both sides of \eqref{eq representation} completes the proof. 
\end{proof} 
The next theorem concerns the weak error between \eqref{eq:MVSDE} and \eqref{eq:Euler}.
\begin{theorem}  \label{weak error Mckean euler}
Suppose that $b \in \cM_2 \big( \bR^d  \times \mathcal{P}_2(\bR^d) \big) $ , $\Phi \in \cM_2 \big(  \mathcal{P}_2(\bR^d) \big) $ and that $\sigma$ is constant. Then the weak error in the particle approximation with Euler scheme satisfies
\begin{eqnarray}
&& \Big| \bE[\Phi(\eulerlaw[T])] - \Phi (\mclaw[T])   \Big|
 \leq C \Big( \frac{1}{N} +h \Big),  \label{eq:expansionformulaeuler} 
\end{eqnarray}
where $C$ is a constant that depends on $\Phi$, $b$, $\sigma$ and $T$, but does not depend on $N$ or $h$. 
\end{theorem}
\begin{proof}
The main idea of the proof is identical to the previous theorem, with the extra complication of time discretisation. Let $\mathbf Z^{N,h}:= (Z^{1,N,h}, \ldots, Z^{N,N,h})$. As before, by Lemma 2.5 and Theorem 2.11 from \cite{chassagneux2019weak},
\begin{equation*} \big| \bE\big( {\cV}(0,\eulerlaw[0]) - \cV(0,\nu) \big) \big| \leq \frac{C}{N}. 
\end{equation*} 
Since $\sigma$ is constant, we adopt the existing convention that the diffusion matrix $a$ is defined by $a = \sigma \sigma^T$. Next, by the previous analysis, we observe that 
\begin{eqnarray} 
&& \big( \Phi(\eulerlaw[T]) - \Phi (\mclaw[T]) \big) - \big( {\cV}(0,\eulerlaw[0]) - \cV(0,\nu) \big) \nonumber \\
&= &  {\cV}(T,\eulerlaw[T]) - {\cV}(0,\eulerlaw[0]) \nonumber \\   
& = &   V({T}, \mathbf Z^{N,h}_{T}) - V(0, \mathbf{Z}^{N,h}_0) \nonumber \\
& = & \bigg[ \int_0^{T} \frac{ \partial V}{ \partial s} (s, \mathbf Z^{N,h}_s) + \sum_{i=1}^N \frac{\partial V}{\partial x_i} ( s, \mathbf Z^{N,h}_s)  b \big( Z^{i,N,h}_{\eta(s)}, \etaeulerlaw[s] \big)  
+ \frac{1}{2}  \text{Tr} \bigg( a  \sum_{i=1}^N \frac{\partial^2 V}{\partial x^2_i}  (s,\mathbf Z^{N,h}_s) \bigg) \,ds \bigg] \nonumber \\
&& +  \int_0^{T}  \sum_{i=1}^N \frac{\partial V}{\partial x_i} ( s, \mathbf Z^{N,h}_s)^T \sigma \, dW_s^i \nonumber \\
& = &  \int_0^{T}  \partial_s {\cV} \big( s, \eulerlaw[s]  \big) + \sum_{i=1}^N \Bigg[ \frac{1}{N} \partial_{\mu} {\cV} \big( s, \eulerlaw[s]  \big)( {Z^{i,N,h}_s} )  b \big( Z^{i,N,h}_{\eta(s)}, \etaeulerlaw[s] \big)    \nonumber \\
&& \, \, + \frac{1}{2}  \text{Tr} \Bigg( a   \bigg( \frac{1}{N} \partial_{v}   \partial_{\mu} {\cV} \big( s, \eulerlaw[s] \big)  ({Z^{i,N,h}_s} )  
+ \frac{1}{N^2} \partial^2_{\mu}  {\cV} \big( s, \eulerlaw[s] \big)  ({Z^{i,N,h}_s} , {Z^{i,N,h}_s} ) \bigg) \Bigg) \Bigg] \,ds \nonumber \\
&& +  \int_0^{T}  \sum_{i=1}^N \frac{1}{N} \partial_{\mu} {\cV} \big( s, \eulerlaw[s]  \big) ( {Z^{i,N,h}_s})^T \sigma \, dW_s^i \nonumber \\
& = &  \int_0^{T} \sum_{i=1}^N \Bigg[ \frac{1}{N} \partial_{\mu} {\cV} \big( s, \eulerlaw[s]  \big)( {Z^{i,N,h}_s} ) \big( b \big( Z^{i,N,h}_{\eta(s)}, \etaeulerlaw[s]  \big) - b \big( Z^{i,N,h}_{s}, \eulerlaw[s]  \big) \big) \nonumber  \\
&& \, \, + \frac{1}{2}  \text{Tr} \Bigg( a  \frac{1}{N^2} \partial^2_{\mu}  {\cV} \big( s, \eulerlaw[s] \big)  (Z^{i,N,h}_s, Z^{i,N,h}_s) \Bigg) \Bigg] \,ds +  \int_0^{T}  \sum_{i=1}^N \frac{1}{N} \partial_{\mu} {\cV} \big( s, \eulerlaw[s] \big) ( Z^{i,N,h}_s)^T \sigma \, dW_s^i. \nonumber \\
&& \label{eq:PDEcancellationeuler}
\end{eqnarray}
The second term of \eqref{eq:PDEcancellationeuler} clearly converges to zero in the rate $O(1/N)$ upon taking expectation. The third term  of \eqref{eq:PDEcancellationeuler} becomes zero upon taking expectation. It remains to deal with the first term  of \eqref{eq:PDEcancellationeuler}. 
Let $\{ \cF_t \}_{t \in [0,T]} $ be the filtration generated by $W^1, \ldots, W^N$. Then, by the It\^{o}'s formula, for each $k \in \{1, \ldots, d \}$,
\begin{eqnarray}
&& \bE \Big[ b_k \big( Z^{i,N,h}_{\eta(s)}, \etaeulerlaw[s]  ) - b_k \big( Z^{i,N,h}_{s}, \eulerlaw[s]  \big) \Big| \cF_{\eta(s)} \Big]  \nonumber \\
& = & - \bE \bigg[ \int_{\eta(s)}^s \bigg( \partial_x b_k( Z^{i,N,h}_{r}, \eulerlaw[r]  ) + \frac{1}{N} \pmu b_k( Z^{i,N,h}_{r}, \eulerlaw[r]   \big)(Z^{i,N,h}_r)\bigg) \cdot dZ^{i,N,h}_{r} \nonumber \\
&& + \sum_{j \neq i} \int_{\eta(s)}^s \frac{1}{N} \pmu b_k( Z^{i,N,h}_{r}, \eulerlaw[r]   \big)(Z^{j,N,h}_r)  \cdot dZ^{j,N,h}_{r} \nonumber \\
&& + \int_{\eta(s)}^s \text{Tr} \bigg( \bigg( \partial^2_x b_k(Z^{i,N,h}_r, \eulerlaw[r]) + \frac{2}{N}  \partial_x \pmu b_k(Z^{i,N,h}_r, \eulerlaw[r])(Z^{i,N,h}_r) \nonumber \\
&& + \frac{1}{N} \partial_v \pmu b_k(Z^{i,N,h}_r, \eulerlaw[r])(Z^{i,N,h}_r)  + \frac{1}{N^2} \ptwomu b_k(Z^{i,N,h}_r, \eulerlaw[r])(Z^{i,N,h}_r,Z^{i,N,h}_r) \bigg) \,d\lev Z^{i,N,h} \rev_{r} \bigg) \nonumber \\
&& + \sum_{j \neq i} \int_{\eta(s)}^s \text{Tr} \bigg( \bigg( \frac{1}{N} \partial_v \pmu b_k(Z^{i,N,h}_r, \eulerlaw[r])(Z^{j,N,h}_r) \nonumber \\
&& + \frac{1}{N^2} \ptwomu b_k(Z^{i,N,h}_r, \eulerlaw[r])(Z^{j,N,h}_r,Z^{j,N,h}_r) \bigg) \,d\lev Z^{j,N,h} \rev_{r} \bigg) \bigg| \cF_{\eta(s)} \bigg]  \nonumber \\
 & = & - \bE \bigg[ \int_{\eta(s)}^s  \sum_{j=1}^N   \frac{1}{N} \pmu b_k( Z^{i,N,h}_{r}, \eulerlaw[r]   \big)(Z^{j,N,h}_r) \,b(Z^{j,N,h}_{\eta(r)}, \etaeulerlaw[r]) \,dr \nonumber \\
 && +  \int_{\eta(s)}^s  \partial_x b_k( Z^{i,N,h}_{r}, \eulerlaw[r]   \big) \,b(Z^{i,N,h}_{\eta(r)}, \etaeulerlaw[r]) \, dr  \nonumber \\
 && + \sum_{j=1}^N  \int_{\eta(s)}^s \frac{1}{N} \pmu b_k( Z^{i,N,h}_{r}, \eulerlaw[r]   \big)(Z^{j,N,h}_r)^T \,\sigma \,dW^j_r +   \frac{1}{N} \int_{\eta(s)}^s \partial_x b_k( Z^{i,N,h}_{r}, \eulerlaw[r]   \big)^T \,\sigma \, dW^i_r  \nonumber \\
&& + \sum_{j=1}^N  \int_{\eta(s)}^s \text{Tr} \bigg( a \bigg( \frac{1}{N} \partial_v \pmu b_k(Z^{i,N,h}_r, \eulerlaw[r])(Z^{j,N,h}_r) \nonumber \\
&& + \frac{1}{N^2} \ptwomu b_k(Z^{i,N,h}_r, \eulerlaw[r])(Z^{j,N,h}_r,Z^{j,N,h}_r) \bigg)  \bigg) \, dr \nonumber \\
&& +  \int_{\eta(s)}^s \text{Tr} \bigg( a \bigg( \partial^2_x b_k(Z^{i,N,h}_r, \eulerlaw[r]) +  \frac{2}{N}  \partial_x \pmu b_k(Z^{i,N,h}_r, \eulerlaw[r])(Z^{i,N,h}_r)  \bigg)  \bigg) \, dr  \bigg| \cF_{\eta(s)} \bigg]  \nonumber \\
& = & - \int_{\eta(s)}^s  \bE \bigg[ \sum_{j=1}^N  \frac{1}{N} \pmu b_k( Z^{i,N,h}_{r}, \eulerlaw[r]   \big)(Z^{j,N,h}_r) \,b(Z^{j,N,h}_{\eta(r)}, \etaeulerlaw[r])  \nonumber \\
 && +  \partial_x b_k( Z^{i,N,h}_{r}, \eulerlaw[r]   \big) \,b(Z^{i,N,h}_{\eta(r)}, \etaeulerlaw[r])  \nonumber \\
&& + \sum_{j=1}^N  \text{Tr} \bigg(a \bigg( \frac{1}{N} \partial_v \pmu b_k(Z^{i,N,h}_r, \eulerlaw[r])(Z^{j,N,h}_r) + \frac{1}{N^2} \ptwomu b_k(Z^{i,N,h}_r, \eulerlaw[r])(Z^{j,N,h}_r,Z^{j,N,h}_r) \bigg) \bigg)  \nonumber \\
&& +  \text{Tr} \bigg( a \bigg( \partial^2_x b_k(Z^{i,N,h}_r, \eulerlaw[r]) +  \frac{2}{N}  \partial_x \pmu b_k(Z^{i,N,h}_r, \eulerlaw[r])(Z^{i,N,h}_r)  \bigg)  \bigg)  \bigg| \cF_{\eta(s)} \bigg]  \, dr. \nonumber \\ && \label{eq:conditional expectation arg weak} \end{eqnarray}
Hence, upon taking expectation, by \eqref{eq:conditional expectation arg weak}, the first term of \eqref{eq:PDEcancellationeuler} can be rewritten as
\begin{eqnarray}
&& \int_0^{T} \sum_{i=1}^N \bE \bigg[ \frac{1}{N} \partial_{\mu} {\cV} \big( s, \eulerlaw[s]  \big)( {Z^{i,N,h}_{s}} ) \big( b \big( Z^{i,N,h}_{\eta(s)}, \etaeulerlaw[s]  \big) - b \big( Z^{i,N,h}_{s}, \eulerlaw[s]  \big) \big) \bigg] \,ds \nonumber \\
& = & \int_0^{T} \frac{1}{N} \sum_{i=1}^N \sum_{k=1}^d \bE  \bigg[ \Big( \partial_{\mu} {\cV}  \big( s, \eulerlaw[s]  \big)( {Z^{i,N,h}_{s}} )  \Big)_k \bE \bigg[ \big( b_k \big( Z^{i,N,h}_{\eta(s)}, \etaeulerlaw[s]  \big) - b_k \big( Z^{i,N,h}_{s}, \eulerlaw[s]  \big) \big) \bigg| \cF_{\eta(s)} \bigg] \bigg]\,ds  \nonumber \\
& = & - \int_0^{T} \int_{\eta(s)}^s \frac{1}{N} \sum_{i=1}^N \sum_{k=1}^d \bE \bigg[  \Big( \partial_{\mu} {\cV}  \big( s, \eulerlaw[s]  \big)( {Z^{i,N,h}_{s}} )  \Big)_k \times \nonumber \\
&&  \bigg[ \sum_{j=1}^N  \frac{1}{N} \pmu b_k( Z^{i,N,h}_{r}, \eulerlaw[r]   \big)(Z^{j,N,h}_r) \,b(Z^{j,N,h}_{\eta(r)}, \etaeulerlaw[r])   +   \partial_x b_k( Z^{i,N,h}_{r}, \eulerlaw[r]   \big) \,b(Z^{i,N,h}_{\eta(r)}, \etaeulerlaw[r])  \nonumber \\
&& + \sum_{j=1}^N  \text{Tr} \bigg( a \bigg( \frac{1}{N} \partial_v \pmu b_k(Z^{i,N,h}_r, \eulerlaw[r])(Z^{j,N,h}_r) + \frac{1}{N^2} \ptwomu b_k(Z^{i,N,h}_r, \eulerlaw[r])(Z^{j,N,h}_r,Z^{j,N,h}_r) \bigg) \bigg)  \nonumber \\
&& +  \text{Tr} \bigg( a \bigg( \partial^2_x b_k(Z^{i,N,h}_r, \eulerlaw[r]) +  \frac{2}{N}  \partial_x \pmu b_k(Z^{i,N,h}_r, \eulerlaw[r])(Z^{i,N,h}_r)  \bigg) \bigg) \bigg] \bigg] \,dr \,ds. \nonumber 
\end{eqnarray}
Finally, by \eqref{eq:boundedinL2} and the fact that $\cV \in \cM_2([0,T] \times \cP_2(\bR^d))$, we have
$$ \bigg| \int_0^{T} \sum_{i=1}^N \bE \bigg[ \frac{1}{N} \partial_{\mu} {\cV} \big( s, \eulerlaw[s]  \big)( {Z^{i,N,h}_{s}} ) \big( b \big( Z^{i,N,h}_{\eta(s)}, \etaeulerlaw[s]  \big) - b \big( Z^{i,N,h}_{s}, \eulerlaw[s]  \big) \big) \bigg] \,ds \bigg| \leq Ch.$$
\end{proof}
\end{subappendices}

\end{document}